\documentclass[reqno,12pt]{amsart}
\usepackage{geometry}
\usepackage{amsmath,amssymb,mathrsfs,color}
\usepackage[dvipdfmx,colorlinks,
linkcolor=red,
anchorcolor=green,
citecolor=blue, 
]{hyperref}
\usepackage{subfigure}
\usepackage{float} 
\usepackage{times}
\usepackage{tikz}
\usetikzlibrary{intersections}
\usepackage{paralist}

\usepackage{comment}

\makeatletter
\def\setliststart#1{\setcounter{\@listctr}{#1}%
  \addtocounter{\@listctr}{-1}}
\makeatother

\makeatletter
\@addtoreset{figure}{section}
\makeatother

\usepackage{calc}
 \newtheorem{The}{Theorem}[section]
 
 \newtheorem{Lem}[The]{Lemma}
 \newtheorem{Pro}[The]{Proposition}
 \theoremstyle{definition}

 \newtheorem{ex}[The]{Example}
 \numberwithin{equation}{section}

\newcounter{Mr}
\newtheorem{Result}[Mr]{\textbf{Main Result}} 

\newcommand{\T}{\mathbb{T}}
\newcommand{\R}{\mathbb{R}}

\newcommand{\N}{\mathbb{N}}

\def\erf{\eqref}
\def\Lip{\operatorname{Lip}}
\def\du#1{\langle#1\rangle}
\def\tim{\times} \def\cS{\mathcal{S^-}}\def\mid{\,:\,}
\def\cA{\mathcal{A}}
\def\leq{\leqslant}\def\geq{\geqslant}
\def\gl{\lambda} \def\frM{\mathfrak{M}} \def\bproof{\begin{proof}} 
\def\eproof{\end{proof}}\def\ep{\varepsilon}\def\gd{\delta}\def\dist{\operatorname{dist}}
\def\cP{\mathcal{P}}\def\bald{\begin{aligned}}
\def\eald{\end{aligned}}\def\beq{\begin{equation}} \def\eeq{\end{equation}}
\def\pl{\partial}
\def\gs{\sigma}\def\gf{\varphi} \def\ol{\overline} 
\def\fr{\frac} \def\gth{\theta}
\def\stm{\setminus}
\def\Gth{\Theta}\def\gk{\kappa}\def\ga{\alpha}

\title{Vanishing contact structure problem and convergence of the viscosity solutions}
\author{Qinbo Chen \and Wei Cheng \and  Hitoshi Ishii \and Kai Zhao}
\address{Morningside center of Mathematics, Academy of Mathematics and Systems Science, Chinese Academy of Sciences, Beijing 100190, China}
\email{chenqb@amss.ac.cn}
\address{Department of Mathematics, Nanjing University, Nanjing 210093, China}
\email{chengwei@nju.edu.cn}
\address{Institute for Mathematics and Computer Science, Tsuda University 2-1-1 Tsuda, Kodaira, Tokyo 187-8577, Japan}
\email{hitoshi.ishii@waseda.jp}
\address{Department of Mathematics, Nanjing University, Nanjing 210093, China}
\email{15251879427@163.com}
\date{\today}
\subjclass[2010]{35F21, 49L25, 37J50}
\keywords{Hamilton-Jacobi equation, convergence, vanishing discount, vanishing contact structure, Weak KAM theory, Aubry-Mather theory}
\begin{document}

\begin{abstract}
This paper is devoted to study the vanishing contact structure problem which is a generalization of the vanishing discount problem. Let $H^\lambda(x,p,u)$ be a family of  Hamiltonians of contact type with parameter $\lambda>0$ and  converges to $G(x,p)$. For the contact type Hamilton-Jacobi equation with respect to $H^\lambda$, we prove that, under mild assumptions, the associated viscosity solution $u^{\lambda}$ converges to a specific viscosity solution $u^0$ of the vanished contact equation. As applications, we  give some convergence results for the nonlinear vanishing discount problem.
\end{abstract}
\maketitle

\section{Introduction and main results}\label{section_Int}
Let $M$ be a connected and compact smooth manifold without boundary, equipped with a smooth Riemannian metric $g$, the associated Riemannian distance on $M$ will be denoted by $d$. Let $TM$ and $T^*M$ denote the tangent and cotangent bundles respectively. A point of $TM$ will be denoted by $(x,v)$ with $x\in M$ and $v\in T_xM$, and a point of $T^*M$ by $(x, p)$ with $p\in  T_x^*M$ is a linear form on the vector space $T_xM$.   With a slight abuse of notation, we shall both denote by $|\cdot|_x$ the norm induced by $g$ on the fiber $T_xM$ and also the dual norm on $T_x^*M$. Let   $H^\lambda(x,p,u):T^*M\times\R\to\R$ be  a family of  Hamiltonians with parameter $\lambda\in(0,1]$ and  $G:T^*M\to\R$ be a Hamiltonian such that $H^\lambda$  uniformly converges, as $\lambda\to 0$, to $G$ on any compact subsets.  

Consider the contact type Hamilton-Jacobi equation 
\begin{equation}\label{Int_lambda}
	H^\lambda\big(x,Du(x),u(x)\big)=c,  \quad x\in M.
\end{equation}
This equation, from the view of physics,  naturally arises in contact Hamiltonian mechanics (see for instance \cite{Bravetti2017, Bravetti_Cruz_Tapias2017}).  Systematic discussions of the  contact transformations and Hamilton-Jacobi equations can be found in  \cite{Giaquinta_Hildebrandt_book_II}.
 Since  $H^\lambda$ uniformly converges to $G$,  a natural and important question is the convergence of viscosity solutions  of \eqref{Int_lambda}, that is, whether or not the whole family $\{u^\lambda\}_\lambda$  of viscosity  solutions uniformly converges, as $\lambda\to 0$, to a unique function. By the stability property of viscosity solution, it is known that the limits of $\{u^\lambda\}_\lambda$ must be viscosity solutions of 
\begin{equation}\label{Int_0}
	G\big(x,Du(x)\big)=c,  \quad x\in M.
\end{equation}

An immediate comment to be made is that, as is well known,
there is only one special \emph{critical value} $c=c(G)$ such that \eqref{Int_0} admits a  viscosity solution.  It is thus natural to take the value $c=c(G)$ in  equations \eqref{Int_lambda} and \eqref{Int_0}. This leads us to consider  the following contact type Hamilton-Jacobi equations
\begin{equation}\label{eq_H_lambda}\tag{HJ$_\lambda$}
H^\lambda\big(x,Du(x),u(x)\big)=c(G),\quad x\in M.
\end{equation}
and 
\begin{equation}\label{eq_G}\tag{HJ$_0$}
G\big(x,Du(x)\big)=c(G),\quad x\in M.
\end{equation}

Our main goal in this paper is to study the \emph{vanishing contact structure problem}, which mainly 
focuses on the asymptotic behavior of the whole family $\{u^\lambda\}_\lambda$ of solutions of \eqref{eq_H_lambda} as $\lambda$ goes to zero, and the characterization of the possible limits. 
This problem is a natural generalization of the \emph{vanishing discount problem} 
where $H^\lambda$ is a linear discounted  system $H^\lambda(x,p,u)=\lambda u+G(x,p)$. 
The vanishing discount problem is also called \emph{ergodic approximation} in PDE, 
which was first studied in a general framework by Lions, Papanicolaou, and Varadhan in \cite{Lions_Papanicolaou_Varadhan1987} 
and  has been widely studied since then. Recently, great progress has been made 
in the vanishing discount problem under various type of settings and methods, see  \cite{Gomes2008, Iturriaga_Sanchez-Morgado2011, DFIZ2016, DFIZ_Math_Z_2016, 
AAIY2016, Soga2017, Mitake_Tran2017, Ishii_Mitake_Tran2017_1, Ishii_Mitake_Tran2017_2, Gomes_Mitake_Tran2018}, etc, especially the results in \cite{DFIZ2016}  where  the authors first prove a convergence result under very mild conditions, and  characterize the unique limit in terms of Peierls barrier and projected Mather measures from a dynamical viewpoint.

However, the vanishing contact structure problem has not been fully studied yet, as far as we know.  The difficulties are that the  general $H^{\lambda}(x, p, u)$ might not be linear in $u$ as discounted case, and  hence one can not directly obtain a convenient explicit representation formula for the solution $u^\lambda$. In recent works \cite{Su_Wang_Yan2016, Wang_Wang_Yan2017, Wang_Wang_Yan2018}, certain implicit variational principle is applied to give representation formula for the viscosity solutions or weak KAM solutions of \eqref{eq_H_lambda}. An alternative approach following Herglotz' generalized variational principle is also obtained  from the Lagrangian formalism  \cite{CCWY2018}, which is later used to obtain a vanishing contact structure result on relavent Cauchy problem for evolutionary equations \cite{Zhao_Cheng2018}. But, these results are established under $C^2$ Tonelli conditions, so they are no longer applicable  to our settings in this paper. Therefore, in this paper, we will develop some new techniques, under suitable assumptions, to handle the vanishing contact structure problem. 

\medskip
Now we begin to state the main assumptions here. Suppose that  
\begin{enumerate}[\rm(SH1)]
    \item $H^\lambda\in C(T^*M\times\R)$ and  $H^\lambda(x,p,0)=G(x, p)$ for each $\lambda$.
	\item The map $p\mapsto G(x,p)$ is convex on $T^*_xM$, for every $x\in M$. 
	\item $G(x,p)$ is coercive in the fibers, i.e.  $\lim_{|p|_x\to +\infty} G (x, p)=+\infty$, uniformly for all $x\in M.$
	\item The map $u\mapsto H^\lambda(x,p, u)$ is strictly increasing, and  for any $r>0$, there exists  a constant $\kappa_r^\lambda>0$ such that $\lim_{\lambda\to 0}\kappa_r^\lambda=0$ and
	 \begin{align}\label{assumption_lipschitz}
	 	  |H^\lambda(x, p, u) -H^\lambda(x, p, 0) | \leqslant \kappa_r^\lambda |u|, \quad \text{for all~}  |u|\leqslant r. 
\end{align}
	\item	There exist positive constants  $\delta_R^\lambda\leqslant  K^\lambda_R$ depending on $R$ and $\lambda$ such that
	 \begin{align*}
	 	 \delta_R^\lambda |u| \leqslant |H^\lambda(x, p, u) -H^\lambda(x, p, 0) | \leqslant K_R^\lambda |u| \quad \text{for all~} |p|_x \leqslant R, |u|\leqslant R 
\end{align*}
and for a suitably large  $R_0$,
\begin{align}\label{steady_increasing}
	 \lim_{\lambda\to 0}\frac{\delta^\lambda_{R_0}}{K^\lambda_{R_0}}=1.
\end{align}
\end{enumerate}
Under hypotheses (SH1) and (SH4), it is clear that $H^\lambda$ uniformly converges, as $\lambda\to 0$, to $G$ on any compact subsets of $T^*M\times\R$.  

Notice that assumptions (SH4) and (SH5) together implies that $\delta_R^\lambda\leq \kappa_R^\lambda$ and $\lim_{\lambda\to 0}\delta_{R_0}^\lambda=
\lim_{\lambda\to 0}K_{R_0}^\lambda=0$.

Here we give some notes and remarks on  our assumptions: throughout this paper the Hamiltonians are  required to be only continuous, in such a setting, of course, no Hamiltonian dynamics can necessarily be defined in the usual sense. $H^\lambda(x,p, 0)\equiv G(x, p)$ in assumption (SH1) is necessary because otherwise, the convergence result might not hold, see Example \ref{ex_nonuni}.  
The coercivity in assumption (SH3) is used to prove that each subsolution of \eqref{eq_H_lambda} should be Lipschitz continuous. The strict monotonicity assumption (SH4), as is known to all, guarantees the uniqueness of Lipschitz continuous viscosity solution of \eqref{eq_H_lambda}, and the local Lipschitz continuity \eqref{assumption_lipschitz} is used to give uniform estimates for the whole family of solutions $\{u^\lambda\}_\lambda$. 
As for assumption (SH5), the constant $R_0$ in \eqref{steady_increasing} will be specified in  the proof of Theorem \ref{The_uni_sol}, this hypothesis might seem too strict at first sight, however, in the sequel (see Main Results 2 and 3 below) we will show that besides the well known discounted system, it is really satisfied by a large class of models. 
Last but not least, to characterize the limit solution more precisely, we make a crucial use of  the convexity assumption (SH2) and the coercivity assumption (SH3) to apply  Aubry-Mather theory.

We denote by $\mathfrak{M}(L_G)$  the set of all projected Mather measures with respect to $L_G$ (see 
Section \ref{section_pre} for the precise definition), 
and by $\mathcal{F}(G)$  the set of all viscosity subsolutions $w$ of \eqref{eq_G} satisfying
\begin{align*}
	\int_{M} w(x)d\nu(x)\leqslant 0,\quad \forall~\nu\in\mathfrak{M}(L_G).
\end{align*}
Then we address the main results in this paper:
\begin{Result}
Let $\{H^\lambda\}_{\lambda\in(0,1]}$  satisfy \mbox{\rm{(SH1)--(SH5)}}. Then  equation \eqref{eq_H_lambda} has a  unique continuous viscosity solution $u^\lambda$ which is also Lipschitzian, and 
the convergence  
\begin{equation}\label{result1_limit}
	u^0(x)=\lim_{\lambda\to 0}u^\lambda(x) \quad \text{uniformly for all~} x\in M
\end{equation} 
holds for some function $u^0\in\Lip(M)$.  Furthermore, the limit function $u^0$
is a viscosity solution of \eqref{eq_G} and is characterized by formula 

\begin{equation}\label{representations}
	u^0(x)=\sup\limits_{u\in\mathcal{F}(G)} u(x)=\min\limits_{\mu\in \mathfrak{M}(L_G)}\int_{M}h(y,x)\ d\mu(y),
\end{equation}
where  $h(y,x)$ is the Peierls barrier of the Lagrangian $L_G$.
\end{Result}

\medskip
It should be noted that unlike \cite{DFIZ2016}, 
we do not use the superlinearity assumption of  $G$ in the fibers, to obtain the representation formula \eqref{representations}. 
In this respect, we remark that the techniques for dealing with  
lower semicontinuous Lagrangians in \cite{AAIY2016} allows us to 
conclude \eqref{representations}.    

Our assumptions (SH1)--(SH5) 
are true of many models. In the sequel, as applications, we  
give some generalizations of discounted systems which are nonlinear in $u$, and naturally satisfy the aforementioned assumptions. Some of the readers might more interested in the following one:

\medskip

\noindent\emph{Application I:} 
We consider  a direct  generalization of the discounted equations. Suppose that 
\begin{enumerate}[\rm(C1)] 
\item   $H\in C(T^*M)$ is convex and coercive in the fibers.
\item  $f(x, u)\in C^1(M\times\R)$ with  $f_u>0$ and 
\begin{equation*}
	\limsup_{|u|\to 0}\frac{|f_u(x,u)-f_u(x,0)|}{|u|}<+\infty\quad \text{uniformly for~}  x\in M.
\end{equation*}
\end{enumerate} 
Let $c=c(G)$ be the critical value of the Hamiltonian $G(x,p):=f(x,0)+H(x,p).$
Now for each $\lambda>0$, we consider the following Hamilton-Jacobi equations:
\begin{equation}\label{main3_lambda}
f(x,\lambda u)+H(x,Du)=c, \quad~x\in M.
\end{equation}
 and 
\begin{equation}\label{main3_0}
 f(x,0)+H(x,Du)=c, \quad~x\in M.
 \end{equation} 
\begin{Result}
Under the above assumptions \mbox{\rm(C1)--(C2)},  equation \eqref{main3_lambda} has a unique continuous  viscosity solution  $u^\lambda$, which is also Lipschitz continuous, 
and, for some $u^0\in \Lip(M)$, which is indeed a viscosity solution
 of \eqref{main3_0},  the following convergence holds: 
\begin{equation*}
	u^0(x)=\lim_{\lambda\to 0}u^\lambda(x)\quad \text{uniformly for all~} x\in M. 
\end{equation*} 
Moreover, $u^0$ is characterized by 
	\begin{align*}
		u^0(x)=\sup\limits_{w\in\mathcal{FR}(G)} w(x)=\min\limits_{\mu\in \mathfrak{M}(G)}\frac{\int_{M}f_u(y,0)h(y,x)\ d\mu(y)}{\int_{M}f_u(y,0)\ d\mu(y)},
	\end{align*}   
	where $h(y,x)$ is the Peierls barrier of the Lagrangian $L_G$ and $\mathcal{FR}(G)$ is the set of  viscosity subsolutions $w$ of \eqref{main3_0} satisfying $\int_{M} f_u(y,0)w(y)d\mu(y)\leqslant 0$ for all $\mu\in\mathfrak{M}(L_G)$.  
\end{Result}

Similar discussions and convergence results of \eqref{main3_lambda} has also been obtained in \cite[Theorem 2.1]{Gomes_Mitake_Tran2018} by using the nonlinear adjoint method. 

However, compared with their result and proof, our assumption are milder and, moreover,  
we deal with the problem by a single application of a comparison theorem (Theorem \ref{The_comp_prin} 
below) combined with now a classical convergence result by Davini et al. \cite{DFIZ2016} (see Theorem \ref{DFIZ_result} below).

\medskip

\noindent \emph{Application II:} 
Now we study a more general class of systems, let $H=H(x,p,u):T^*M\times\R\to \R$ and  $G(x,p):=H(x,p,0)$ be the   Hamiltonian with $c=c(G)$  the critical value of $G$. For each $\lambda>0$, we consider the equations
\begin{equation}\label{gen_eq}
	H(x, Du(x), \lambda u(x))=c,\quad x\in M.
\end{equation}
and 
\begin{equation}\label{gen_crit_eq}
	H(x, Du(x),0)=c,\quad x\in M.
\end{equation}
 Then we can prove
\begin{Result} Suppose that
\begin{enumerate}[\rm(a)]
    \item  $H\in C(T^*M\times\R)$, the partial derivative $H_u\in C(T^*M)$ with  $H_u>0$. For any $R>0$, there exists  $B_R>0$ so that
\begin{equation*}
	|H(x,p,u)-H(x,p,0)|\leqslant B_R|u|\quad \text{for all~} |u|\leqslant R
\end{equation*}
 and
    \begin{equation*}
    	|H_u(x,p,u)-H_u(x,p,0)|\leqslant B_R|u| \quad\text{for all~} |p|_x\leqslant R, |u|\leqslant R\,;
    \end{equation*}
    \item $G(x,p)$ and $\frac{G(x,p)-c}{H_u(x,p,0)}$ is convex and coercive in the fibers.    
 \end{enumerate} 
Then  equation \eqref{gen_eq} has a unique continuous viscosity solution $u^\lambda$, which is also Lipschitzian,  
and, moreover, for some Lipschitz viscosity solution $u^0$ of \eqref{gen_crit_eq}, 
\begin{equation*}
	u^0(x)=\lim_{\lambda\to 0}u^\lambda(x) \quad \text{uniformly for all~} x\in M. 
\end{equation*} 
Furthermore, $u^0$ is represented as 
\begin{equation}
	u^0(x)=\sup\limits_{u\in\mathcal{F}(\breve{G})} u(x)=\min\limits_{\mu_*\in \mathfrak{M}(L_{\breve{G}})}\int_{M}\breve{h}(y,x)\ d\mu_*(y).
\end{equation}
where  $\breve{h}(y,x)$ is the Peierls barrier of the Lagrangian $L_{\breve{G}}$ associated with the Hamiltonian 
$\breve{G}(x,p)$ $= \frac{G(x,p)-c}{H_u(x,p,0)}.$	
\end{Result}

The  paper is organized as follows. Section \ref{section_pre} provides basic terminologies and notations which are necessary for the understanding of  our subsequent work, and we collect some necessary results 
in Aubry-Mather theory and weak KAM theory under a non-smooth setting and without superlinearity assumption. In Section \ref{section_existence}, we establish the existence and uniqueness of  viscosity solution  for  equation \eqref{eq_H_lambda}, based on Perron's construction and the comparison principle. 
Section \ref{section_main} is the main part of the present paper which consists of the discussion, statement and proof of  Main Result 1. 
Section \ref{section_app} provides some applications and gives the proof of  Main Results 
2 and 3,  which strongly rely on the nonlinear analysis of our models. Finally, in Section \ref{sec6}, for the reader's convenience, we give the detailed proof for Theorem \ref{DFIZ_result} which is crucial in this paper. Appendices A and B  examine briefly the validity of the variational or optimal control formula 
for the solutions of the Cauchy problem for Hamilton-Jacobi equations as well as the existence 
of solutions and the comparison principle. The role of the appendices are to bridge a gap, at least in the literature, in the basic theory of the Hamilton-Jacobi equations with  
Lagrangians having possibly the value $+\infty$.  
\medskip

\noindent\textbf{Acknowledgements.} Qinbo Chen is partially supported by National Natural Science Foundation of China (Grant No. 11631006). Wei Cheng is partially supported by National Natural Science Foundation of China (Grant No. 11871267, No. 11631006 and No. 11790272). Hitoshi Ishii is partially supported by Grant-in-Aid for Scientific Research Grant ((B) No. 16H03948, (S) No. 26220702, (B) No. 18H00833), and he would like to
thank the Department of Mathematics, Nanjing University for its hospitality.

\section{Preliminaries}\label{section_pre}

In this section, we provide  some useful results from Aubry-Mather theory and weak KAM  theory which are necessary for the purpose of this paper. 
Aubry-Mather theory  are classical and  well known for $C^2$ Hamiltonians or Lagrangians satisfying Tonelli conditions, we refer the reader to Mather's original papers \cite{Mather1991, Mather1993} and Ma\~n\'e's papers \cite{Mane1996, Mane1997}.  For a complete introduction to the weak KAM theory under Tonelli settings, we refer the reader to Fathi's book \cite{Fathi_book}. An analogue of Aubry-Mather theory or weak KAM theory has been developed for contact Hamiltonian, see e.g \cite{Maro_Sorrentino2017, Mitake_Kohei2018, Wang_Wang_Yan2018a, Wang_Wang_Yan2018b} . 

However, in this paper our systems are only required to be continuous and are lack of Hamiltonian or Lagrangian dynamics. So we  need the generalizations of Aubry-Mather theory and weak KAM theory for non-smooth systems, the main references are  \cite{Fathi_Siconolfi2005,Ishii2008,AAIY2016, Davini_Zavidovique2013,DFIZ2016,Ishii2011}.

Throughout this section,  $M$ is a connected and compact manifold without boundary, and $H(x,p):T^*M\to\R$ is assumed to satisfy
\begin{enumerate}[\rm(H1)]
    \item $H\in C(T^*M)$.
    \item For every $x\in M$, the map $p\mapsto H(x, p)$ is  convex in the fiber $T^*_x M$.
	\item  $H(x, p)$ is coercive in the fibers, i.e. $\lim_{|p|_x\to+\infty}H(x,p)=+\infty$ uniformly in  $x\in M.$
\end{enumerate}	  
By the Legendre-Fenchel transformation, we define the Lagrangian associated with $H$ by
 \begin{align*}
	L(x,v)=\sup_{p\in T^*_xM}\{\du{p,v}_{x}-H(x,p)\},\quad  \forall~(x,v)\in TM,
\end{align*}
where $\du{p,v}_{x}$ denotes the value of the linear form $p\in T_x^*M$ evaluated at $v\in T_xM$. 

It is a classical result in convex analysis that if, in addition, $H$ has superlinear growth in the fibers, then 
$L\in C(TM)$. But we do not assume here the superlinearity of $H$ in the fibers, which results in a possibility of 
$L$ taking value $+\infty$.  However, by the definition of $L$, it is clear that $L$ is lower 
semicontinous in $TM$. Moreover, $L$ is bounded in a neighborhood of zero section of 
$TM$ and superlinear in the fibers.  Indeed, we observe that 
\beq\label{L-lower-b}
L(x,v)\geq -H(x,0), 
\eeq
Also, since $p\mapsto H(x,p)$ is convex and coercive and $M$ is compact, 
there exist constants $\gd>0$ and $C_0>0$ such that $H(x,p)\geq \gd |p|_x -C_0$. Hence,
\[
\du{p,v}_x-H(x,p)\leq |p|_x|v|_x-\gd|p|_x+C_0 =(|v|_x-\gd)|p|_x+C_0
\]
and therefore
\beq\label{L-upper-b}
L(x,v)\leq C_0 \ \ \ \text{ if }\ |v|_x\leq \gd. 
\eeq

Furthermore, for any nonzero $v\in T_xM$ and $R>0$, by definition,
\beq\label{L-superlinear}
L(x, v)\geq \max_{|p|_x=R}\langle p, v\rangle_x-\max_{|p|_x=R}H(x,p)=R|v|_x-\max_{|p|_x=R}H(x,p).
\eeq
Since $R>0$ is arbitrary, the Lagrangian $L$ has a superlinear growth in the fibers.

The {\em critical value} $c(H)$ associated with $H$ is defined by 
\begin{equation}
	c(H)=\inf\{c\in\R ~:~ H(x,Du)=c\ \text{~admits a viscosity subsolution}\}. 
\end{equation}
It is well known that  the  critical value is the unique real number $c$  such that the equation
\begin{align*}
	H(x,Du)=c,\quad x\in M,
\end{align*}
admits a global viscosity solution (see  \cite{Fathi_Siconolfi2005}).  
The notion of viscosity solution 
has been introduced by Crandall and Lions \cite{Crandall_Lions1983}.   A function $u: V\to\R$ is a \emph{viscosity subsolution} of $H(x, Du)=c$ on the open subset $V \subset M$ if, for every $C^1$ function $\varphi: V\to\R$, with $\varphi \geqslant u$  and $\varphi(x_0)=u(x_0)$, we have $H(x_0, 
D\varphi(x_0))\leqslant c$. It is a \emph{viscosity supersolution} if, for every $C^1$ function $\psi: V\to\R$, with  $\psi\leqslant u$  and $\psi(x_0)=u(x_0)$, we have $H(x_0, D\psi(x_0)) \geqslant c$. A function $u: V \to\R$ is a \emph{viscosity solution}  if it is both a viscosity subsolution and a viscosity supersolution.

The following result is classical in viscosity solution theory (see \cite[Theorem I.14]{Crandall_Lions1983}, 
\cite[Example 1]{Ishii1987}) and we refer the reader  to  \cite[Theorem 5.2]{Fathi2012} for the case  $M$ is a  compact manifold.
\begin{Pro}\label{fathisublip}
	Let $H: T^*M\to\R$ be coercive in the fibers and $c\in\R$, then any continuous viscosity subsolution of $H(x, Du)=c$ is  Lipschitz continuous. Moreover, the Lipschitz constant is bounded above by $\kappa_c$ with
	$$\kappa_c=\sup\{|p|_x : H(x,p)\leq c\}.$$
\end{Pro}

For every $t>0$ and $x, y\in M$, let $\Gamma^t_{x,y}$ denote the set of absolutely continuous curves $\xi: [0,t]\to M$ with $\xi(0)=x$ and $\xi(t)=y$.  We define the {\em action function}
\begin{align*}
	h_t(x,y)=\inf_{\xi\in \Gamma^t_{x,y}}\int_{0}^{t}[L(\xi(s),\dot{\xi}(s))+c(H)]\ ds,
\end{align*}
which might be $+\infty$ if the distance between $x$ and $y$ is large and $t>0$ is small. 
Then, we define a real-valued function $h$ on $M\times M$ by  
\begin{align*}
	h(x,y):=\liminf_{t\to\infty} h_t(x,y).
\end{align*}
In the literature, $h(x,y)$ is called the {\em Peierls barrier}. This leads us to define the so-called {\em projected Aubry set} $\mathcal{A}$  by
\begin{align*}
	\mathcal{A}:=\{x\in M: h(x,x)=0\}.
\end{align*}
Since $M$ is compact, the projected Aubry set $\mathcal{A}$ is nonempty and closed.

We recall the ``semi-distance'' $S : M\tim M\to\R$, as introduced in \cite{Fathi_Siconolfi2005} (see also \cite[Chapter 8]{Fathi_book}), given by
\[
S(y,x)=\sup\{\psi(x)-\psi(y)\mid \psi\in \cS\big(H-c(H)\big)\}.
\]
Here and henceforth, we denote by $\cS(G)$ the set of all viscosity subsolutions of 
$G(x,Du)=0$. 

We collect some basic properties of the function $S$: 
\begin{Pro} \label{semi-distance} 
\hfill
\begin{enumerate}[\rm(1)]
	\item $S$ is finite-valued function and satisfies the triangle inequality: 
     $S(x,y)\leq S(x,z)+S(z,y)$.
     \item If $w$ is a  viscosity subsolution of $H(x, Du)=c(H)$, then $w(y)-w(x)\leqslant S(x,y)$.
     \item For any $y\in M$, the function $x\mapsto S(y,x)$ is a viscosity subsolution of  
$H(x,Du)=c(H)$.
\end{enumerate}
\end{Pro}
The claims (1)--(3) above are easy  to check. Indeed, 
the set $\cS\big(H-c(H)\big)$ is not empty and, by Proposition \ref{fathisublip}, 
it is equi-Lipschitz continuous in $M$. Consequently, $S$ is well defined as a real-valued function 
and Lipschitz continuous in $M\tim M$.  Property (2) above is a simple consequence of the definition of $S$. In particular, since 
$z\mapsto S(x,z)$ is a member of $\cS\big(H-c(H)\big)$, we have $S(x,y)\leq S(x,z)+S(z,y)$. As for property (3), by the stability of viscosity subsolutions under the
sup-operation, it follows that $x\mapsto S(y,x)$ is a viscosity subsolution.

The results in the proposition  below seem to be new, at least for the Hamiltonian $H$ without a superlinear growth in the fibers. The argument in \cite{AAIY2016} is easily adapted to our current setting, we will present the proof in Section \ref{sec6} for the reader's convenience.   
 
\begin{Pro}\label{Pro_rep_h}
$z\in\cA$ if and only if the function $x\mapsto S(z, x)$ is a viscosity solution of the equation $H(x, Du)=c(H)$. In addition, we have the following representation: 
 \begin{equation*}
 	h(x,y)=\inf_{z\in\cA}[S(x,z)+S(z,y)].
 \end{equation*}
\end{Pro}

As a result, we have the following properties for $h$ where item (1) and item (3)  are a direct consequence of Proposition \ref{semi-distance} and \ref{Pro_rep_h}. As for item (2), note that by Proposition \ref{Pro_rep_h}, for any $z\in\cA$, $x\mapsto S(z,x)$ is a 
viscosity solution of $H(x, Du)=c(H)$ in $M$, and we deduce by the stability property of 
viscosity solutions that $x\mapsto h(y,x)=\inf_{z\in\cA}[S(y,z)+S(z,x)]$ is also a viscosity solution. Item (4) is exactly \cite[Theorem 8.5.5]{Fathi_book}.
\begin{Pro}\label{proper_criti_sol}
\begin{enumerate}[\rm(1)]
	\item $h$ is finite valued and Lipschitz continuous.
     \item For any $y\in M$, the function $x\mapsto h(y,x)$  is a viscosity solution of  
$H(x,Du)=c(H)$ and the function $x\mapsto -h(x,y)$  is a viscosity subsolution of  $H(x,Du)=c(H)$.
	\item If $v$ is a  viscosity subsolution of $H(x, Du)=c(H)$, then $v(y)-v(x)\leqslant h(x,y)$.
	\item If $f$ and $g$ are a viscosity subsolution and a viscosity supersolution 
     of $H(x, Du)=c(H)$ and $f\leq g$ on the set $\mathcal{A}$, then $f\leq g$ in $M$.
\end{enumerate}
\end{Pro}

In this paper, we need a well known approximation argument for subsolutions of convex Hamilton-Jacobi 
equations, this technique  is standard and can be found in many places, for instance, in the proof of 
the assertion (d) of \cite[Lemma 2.2]{Evans1992}.  To better applied to our 
problem, we refer the reader to \cite[Theorem 10.6]{Fathi2012}.

\begin{Pro}\label{sm_subapprox}
Let $u: M\to\R$ be a  Lipschitz viscosity subsolution of the equation $H(x, Du)$ $=c(H)$. Then for any  number $\epsilon>0$, we can find a $C^\infty$ function $w: M\to\R$	such that $\|u-w\|_\infty\leqslant \epsilon$, $H(x, Dw)\leqslant c(H)+\epsilon$ for every $x\in M$, and $\|Dw\|_\infty\leqslant\textup{Lip}(u)+1$. 
\end{Pro}

Next, we will  introduce the notion of  Mather measure. For the purpose of this paper, we adopt the equivalent definition originate from Ma\~n\'e \cite{Mane1996}, see also \cite{Fathi_Siconolfi2004}.
Recall that a Borel probability measure $\tilde{\mu}$ on $TM$ is called {\em closed} if it satisfies $\int_{TM} |v|_x\ d\tilde{\mu}<+\infty$ and
 $$\int_{TM} \du{D\varphi,v}_x\ d\tilde{\mu}(x,v)=0,\quad\text{ for all~} \varphi\in C^1(M).$$ 
 Let $\mathcal{P}$ be the set of closed probability measures on $TM$.  
The set is nonempty and it has been shown that the critical value $c(H)$ can be obtained by considering a minimizing problem. More precisely, the following relation holds  (see \cite[Theorem 5.7]{DFIZ2016} 
in the case of $H$ having a superlinear grwoth in the fibers, and \cite{AAIY2016} in the case of 
$H$ with only coercive condition):
\begin{align*}
	-c(H)=\min_{\tilde{\mu}\in\mathcal{P}}\int_{TM}L(x,v)\ d\tilde{\mu}(x,v).
\end{align*}
Here, we remark that $L$ may take the value $+\infty$ and the identity above, in particular, 
requires at least that $L<+\infty$ $\tilde\mu$-almost everywhere.     

For any  probability measure $\tilde{\mu}$ on $TM$, we define the corresponding projected measure on $M$ by $\mu:=\pi_\#\tilde{\mu}$, where $\pi: TM\to M$ is the canonical projection, namely $\mu(A)=\tilde{\mu}(\pi^{-1}(A))$ and, equivalently, 
\begin{equation*}
	\int_M f(x)\,d\mu(x)=\int_{TM} (f\circ\pi)(x,v)\,d\tilde{\mu}(x,v),\quad \text{for every~} f\in C(M).
\end{equation*}

A {\em Mather measure} for the Lagrangian $L$ is a  measure $\tilde{\mu}\in\mathcal{P}$  satisfying the following 
minimizing property:
\begin{align}
	\int_{TM}L(x,v)\ d\tilde{\mu}=-c(H).
\end{align}
We denote by $\widetilde{\mathfrak{M}}(L)$ the set of all Mather measures, and by   
\begin{equation}\label{pro_Ma}
	\mathfrak{M}(L)=\pi_{\#}\widetilde{\mathfrak{M}}(L)
\end{equation} 
the set of  all projected Mather measures.

Now we end up this section with a  remarkable result which is vital in the following sections.
\begin{The}[\cite{DFIZ2016, AAIY2016}]\label{DFIZ_result}
Let $H(x, p)$ satisfy \mbox{\rm(H1)-(H2)} and be coercive in the fibers. Then $$\lambda u+H(x, D u)=c(H), \quad x\in M$$ 
has a unique continuous viscosity solution $u^\lambda$, the family $\{u^\lambda\}$ 
converges uniformly, as $\lambda\to 0$, to a single critical solution $u^0$ of $H(x, D u)=c(H)$, and the limit $u^0$ is represented as 
\begin{equation}\label{dfiz2016_rep}
	u^0(x)=\sup\limits_{u\in\mathcal{F}(H)} u(x)=\min\limits_{\mu\in \mathfrak{M}(L)}\int_{M}h(y,x)\ d\mu(y),
\end{equation}
where $\mathcal{F}(H)$ denotes the set of all viscosity subsolutions $w$ of $H(x, Du)=c(H)$ satisfying $\int_{M} w(x)d\nu(x)\leqslant 0$ for every $\nu\in\mathfrak{M}(L)$.
\end{The}

The convergence result in the first part of the theorem above has been obtained in \cite{DFIZ2016} and the representation formula \eqref{dfiz2016_rep}, under the superlinearity assumption on $H$ in the fibers, is contained in \cite{DFIZ2016}. 
Formula \eqref{dfiz2016_rep} in the general case can be easily obtained by adapting the proof in \cite{AAIY2016} 
to our case, or, in other words, by modifying the argument of \cite{DFIZ2016} with some technicalities from 
\cite{AAIY2016}.   Some of the details of such modifications are explained in Section \ref{sec6} 
for the reader's convenience.

\section{Existence of viscosity solutions and some  
Lipschitz estimates}\label{section_existence}
In this section, we establish the existence and uniqueness of  a solution of  equation \eqref{eq_H_lambda},  and give some uniform estimates for the whole family of solutions. 

We begin with a lemma which describes that  every continuous viscosity subsolution is necessarily Lipschitz continuous.
\begin{Lem}\label{sub_is_Lip}
Let $H^\lambda$ satisfy \mbox{\rm(SH1), (SH3)} and \mbox{\rm(SH4)}, then  any continuous viscosity subsolution $u^\lambda$ of \eqref{eq_H_lambda} is Lipschitz continuous.  
\end{Lem}
\begin{proof}
Since $u^\lambda$ is bounded on $M$, we  set  $$\tilde{H}^\lambda(x, p):=H^\lambda(x, p,-\|u\|_\infty),$$ it is a classical continuous Hamiltonian  which is  coercive in the fibers by assumptions (SH3) and (SH4). So $u^\lambda$ is a continuous viscosity subsolution of $\tilde{H}^\lambda(x, Du)=c(G)$. Then by Proposition \ref{fathisublip}, $u^\lambda$ is Lipschitz continuous.
\end{proof}

Next, we  prove a  version of comparison principle. The proof  below is easily extended to the case of 
general semicontinuous sub- and super-solutions, however,  the continuous version is 
enough for the applications in this paper.

\begin{The}\label{The_comp_prin}
 Let $H(x,p,u):T^*M\times\R\to\R$ be a continuous Hamiltonian which is strictly increasing in $u$ for every $(x,p)\in T_x^*M$. Assume that the equation 
 \begin{equation}\label{H}
   	H(x ,Du(x), u(x))=0, \quad x\in M ,
   \end{equation}
  admits a Lipschitz continuous viscosity solution $w$. Then for any  continuous viscosity subsolution $f$ and any continuous  viscosity supersolution $g$, we have   	
 $f \leqslant g $.
   \end{The}

Although the above comparison theorem is well known, for the reader's convenience, we present here a proof of it. 

\begin{proof} 
We will prove $f\leqslant g$ by  showing  $f\leqslant w$ and $w\leqslant g$. To see this, we only need to prove $f\leqslant w$ since the proof of $w\leqslant g$ is similar.

Let $\hat{x}$ be any maximum point such that $ f(\hat{x})-w(\hat{x})=\max (f-w)$. We need only to show that 
\begin{equation}\label{maxleqzero}
	f(\hat{x})-w(\hat{x})\leqslant 0.
\end{equation}
We select a function $\psi\in C^1(M)$ so that $\psi(\hat x)=0$ and $\psi(x)<0$ for all $x\not =\hat x$. 
Note that the function $f-w+\psi$ has a strict maximum $(f-w)(\hat x)$ at $\hat x$ and that $D\psi(\hat x)=0$.

Let  $D\subset M$ be a closed domain contained in one of the coordinate charts of $M$ and  $\hat{x}\in \text{int}D$. Because of the local nature, we can assume, without loss of generality, that  $D=\overline {B(0,1)}$ is a unit ball in $\R^n$ and $\hat{x}=0$. 

Now let $\alpha>0$, consider the function $\Phi_\alpha:D \times D\to\R$ defined by 
   	$$
   	\Phi_\alpha(x, y):=(f+\psi)(x) -w(y)-\alpha |x-y|^2,
   	$$ 	
and let $(x_\alpha, y_\alpha )\in D$ be a maximum point of $\Phi_\alpha$. From $
   	\Phi_\alpha(x_\alpha ,y_\alpha )\geqslant \Phi_\alpha(x_\alpha ,x_\alpha )$ we get
   \begin{equation*}
   \alpha |x_\alpha-y_\alpha|^2\leqslant w(x_\alpha)-w(y_\alpha)\leqslant \text{Lip}(w)|x_\alpha-y_\alpha|.
   \end{equation*}
    From this, we get 
   \begin{equation}\label{alphabounded}
   	\alpha|x_\alpha-y_\alpha| \leqslant \text{Lip}(w)\quad\text{ and }\quad \lim_{\alpha\to+\infty} (x_{\alpha} -y_{\alpha})=0.
   \end{equation}
   	Hence, by compactness,  there exist a sequence $\{ \alpha_j \}_{j\in\N}$ divergent to $+\infty$ and a point $(\hat{z},\hat{p}) \in D\times\R^n$ such that 
   	$$
   	(x_{\alpha_j}, 2\alpha(x_{\alpha_j} - y_{\alpha_j} ) ) \rightarrow (\hat{z} ,\hat{p} ) \quad 
\text{and }\quad \lim_{j\to+\infty} y_{\alpha_j} = \lim_{j\to+\infty} x_{\alpha_j}=\hat{z}.
   	$$  

Now, since $\Phi_\alpha(x_\alpha,y_\alpha)\geqslant \Phi_\alpha(x,x)=(f+\psi-w)(x)$ for all $x\in D$, we see that 
$$
(f+\psi)(x_\alpha)-w(y_\alpha)\geqslant \Phi_\alpha(x_\alpha,y_\alpha)\geqslant \max_{D}(f+\psi-w),
$$ 
and, moreover, after taking limit for $\alpha=\alpha_j$,  
$$
(f+\psi-w)(\hat z)\geqslant \max_{D}(f+\psi-w). 
$$
This implies that $\hat z=\hat{x}=0$ since $f+\psi-w$ has a strict maximum at $\hat {x}=0$. We may thus assume without loss of generality that $x_{\alpha_j}$ lies in the interior of $D$ for every $j\in\N$. 

Since the function $
f(x)-\big[w(y_{\alpha_j})+\alpha_{j}|x-y_{\alpha_j}|^2-\psi(x)\big]$ of $x$  has a 
maximum at $x_{\alpha_j}$ and the function  
$w(y)-\big[(f+\psi)(x_{\alpha_j})-\alpha_j|y-x_{\alpha_j}|^2\big]$ of $y$ has a minimum at $y_{\alpha_j}$, by the viscosity property of $f$ and $w$, we have 
 $$
   	H\Big(x_{\alpha_j}, 2\alpha_j(x_{\alpha_j}-y_{\alpha_j})
-D\psi(x_{\alpha_j}), f(x_{\alpha_j})\Big)\leqslant 0$$ 
 $$ H\Big(y_{\alpha_j}, 2\alpha_j(x_{\alpha_j}-y_{\alpha_j}), w(y_{\alpha_j})\Big)\geqslant 0.
   	$$
Taking limit in $j$ yields
   	$$
   	H\big(\hat{x},\hat{p}-D\psi(\hat{x}), f(\hat{x})\big)\leqslant 0 \quad \text{and} \quad 
   	H\big(\hat{x},\hat{p}, w(\hat{x})\big) \geqslant 0,
   	$$
and furthermore, since $D\psi(\hat{x})=0$,
\[
H\big(\hat{x},\hat{p},f(\hat{x})\big)\leqslant 0 \leqslant  
H\big(\hat{x},\hat{p},w(\hat{x})\big). 
\]
Now, the strict monotonicity in $u$ of $H(x, p, u)$ implies that $f(\hat{x})\leq w(\hat{x})$. 
Thus, \eqref{maxleqzero} holds.
\end{proof}

Now we can prove the following result:
\begin{The}\label{uniq_equi_Lip}
Let $H^\lambda$ satisfy \mbox{\rm(SH1)--(SH4)}. For each $\lambda\in(0,1]$, equation \eqref{eq_H_lambda} admits a unique continuous viscosity solution $u^\lambda$.  In addition, $u^\lambda$ is Lipschitzian, and the  family $\{u^\lambda: \lambda\in(0,1]\}$ is  equi-bounded and equi-Lipschitzian, namely there exist constants $C_0>0$ and $M_0>0$ which are independent of $\lambda$ such that
\begin{equation}\label{apriori_est}
	\|u^\lambda\|_\infty\leqslant C_0,\quad \textup{Lip}(u^\lambda)\leqslant M_0.
\end{equation}
\end{The}
\begin{proof}

Under our assumptions, by the choice of $c(G)$ and Lemma \ref{fathisublip}, \eqref{eq_G} always admits  a Lipschitz viscosity solution, denoted by $u_0$. Because $H^\lambda(x,p,0)=G(x,p)$,  by the monotonicity assumption (SH4), $u_0^-:=u_0-\|u_0\|_\infty$  and $u_0^+:=u_0+\|u_0\|_\infty$ are, respectively, a Lipschitz viscosity subsolution and a Lipschitz viscosity supersolution of \eqref{eq_H_lambda}. 
 Fix $\lambda$ and set  
\begin{equation}\label{perron_construction}
u^\lambda(x):=\sup\{v(x):u_0^-\leqslant v\leqslant u_0^+, v \text{~is a continuous subsolution of~} \eqref{eq_H_lambda} \}.
\end{equation}

Notice that if $v$ is a subsolution of \eqref{eq_H_lambda} and $u_0^- \leqslant v \leqslant u_0^+$, then 
$$|v(x)|\leqslant C_0:=2\|u_0\|_\infty,\quad \text{for all} x \in M.$$
Without loss of generality, we assume $\mathbf{K}_0:=\max_{\lambda} \kappa_{C_0}^\lambda$ is finite since $\lim_{\lambda\to 0}\kappa_{C_0}^\lambda=0$,  by assumption \eqref{assumption_lipschitz} in (SH4). 
Now that $ v$ is a subsolution of the equation 
$$ G(x, Dv) = a,
	  $$
with $a=2 \mathbf{K}_0C_0+c(G)$,  $v$ is Lipschitz by Lemma \ref{fathisublip}. 
In addition,  the coercivity of $G$ implies that for some constant 
$N_0 >0$,
$$
 (x, p) \in T_x^*M ,|p|_x>N_0  \Rightarrow G(x, p)>a,
$$
this shows that, for every $x\in M$, $v$ is Lipschitz continuous with Lipschitz bound $N_0$ 
in a neighborhood of $x$.  Since $M$ is compact and connected, there is a constant $M_0$ such that  
$\text{Lip}(v)\leqslant M_0$.  

Thus, by Lemma \ref{sub_is_Lip}, the function $u^\lambda$ constructed in \eqref{perron_construction} is also Lipschitz continuous with $\text{Lip}(u^\lambda)$ $\leqslant M_0$.  Since $u^\lambda$ is defined by the Perron method (see \cite{Ishii1987}),  
it is a viscosity solution of \eqref{eq_H_lambda}. This proves the existence of continuous viscosity solutions of  \eqref{eq_H_lambda}, and the uniform estimates \eqref{apriori_est}.

Finally, since we have shown that \eqref{eq_H_lambda} admits a  Lipschitz viscosity solution,   Theorem  \ref{The_comp_prin} implies that  \eqref{eq_H_lambda} has only one  continuous viscosity solution, which finishes our proof.
\end{proof}

\section{Convergence of the  viscosity solutions}\label{section_main}

In this section, we investigate the asymptotic behavior of  the viscosity solution $u^\lambda$ of  equation \eqref{eq_H_lambda}, we aim to  study whether or not the limit of $u^\lambda$ exists as $\lambda\to 0$, and if the limit exists, what the characterization of this limit is.  

We firstly clarify a useful result on the limiting behavior of solutions of \eqref{eq_H_lambda}, 
which is a direct consequence of Theorem  \ref{uniq_equi_Lip}, Ascoli-Arzel\`a theorem, 
and a well known stability property of viscosity solutions (see e.g. 
\cite{Crandall_Evans_Lions1984, Barles_book}).
 
\begin{Pro}[Stability]\label{Vanishing}
Let $H^\lambda$  satisfy \mbox{\rm{(SH1)--(SH4)}} and $u^\lambda$ be the  viscosity solution of   \eqref{eq_H_lambda}, then the family $\{u^\lambda\}_{\lambda\in(0,1]}$  has a uniformly convergent subsequence. 
Moreover, for any  subsequence $\{u^{\lambda_n}\}_{n}$ which converges uniformly to $u^0$ as $\lambda_n\to 0$, $u^0$ is a viscosity solution of  equation \eqref{eq_G}.
\end{Pro}

Now it is natural to ask whether the limit $u^0$ is unique. Before that, we give an  example where the family of solutions does not converge.

\begin{ex}[Non-convergence]\label{ex_nonuni}
Let $G(x,p)\in C(T^*M)$ be  coercive and convex in the fibers. For any continuous viscosity solution $w$ of $G(x,Du)=c(G)$, we choose a family $\{w_{\lambda}\}_\lambda$ of $C^2$ functions 
satisfying $\|w-w_\lambda\|\leqslant \lambda$ and consider the Hamiltonians
$$H^\lambda(x,p,u)=\lambda u+G(x,p)-\lambda w_\lambda(x),\quad \lambda>0$$ 	
Then  $w(x)-\lambda$ and $w(x)+\lambda$ are, respectively, a  subsolution  and a supersolution of the  equation $H^\lambda(x,D u(x),u(x))=c(G)$.
Thus by the comparison principle for discounted equations, 
if $u^\lambda\in C(M)$ is the solution of \eqref{eq_H_lambda}, then 
$$w(x)-\lambda\leqslant u^\lambda(x)\leqslant w(x)+\lambda .$$ As $\lambda$ goes to zero, $u^\lambda$ converges to $w$.

Now  let $f$ and $g$ be two distinct viscosity solutions of the equation $G(x,Du)$ $=c(G)$. We select two families $\{f_{\lambda}\}_\lambda$ and  $\{g_{\lambda}\}_\lambda$ of $C^2$ functions satisfying $\|f-f_\lambda\|\leqslant \lambda$ and $\|g-g_\lambda\|\leqslant \lambda$. Consider the following family of Hamiltonian: 
\begin{equation*}
H^\lambda(x,p,u)=
\begin{cases}
\lambda u+G(x,p)-\lambda f_\lambda(x),  \quad &\text{if~} \lambda\in (\frac{1}{2n+1},\frac{1}{2n}], \\
\lambda u+G(x,p)-\lambda g_\lambda(x),\quad &\text{if~} \lambda\in (\frac{1}{2n+2},\frac{1}{2n+1}] .		
\end{cases}
\end{equation*}	
It is easy to check that $H^\lambda$ satisfies  all requirements in assumptions (SH1)--(SH5) except the hypothesis $H^\lambda(x,p,0)\equiv G(x,p)$. However, $H^\lambda(x,p,0)$ converges uniformly to  $G(x,p)$. By the same argument as above, one can find that,  as $\lambda\to 0$, the limits of  $\{u^\lambda\}_\lambda$ cannot be  unique, one is $f(x)$, the other is $g(x)$.
\end{ex}

Under assumptions (SH2) and (SH3), we can define  the  Lagrangian 
 $L_G: TM\to (-\infty,\,+\infty]$ associated with the Hamiltonian $G$ by
\begin{equation}
	L_G(x,v)=\sup_{p\in T_x^*M}\{\du{p,v}_x-G(x,p)\}.
\end{equation} 
We denote by $\mathfrak{M}(L_G)$  the set of all projected Mather measures, associated with $L_G$, and by $\mathcal{F}(G)$  the set of all viscosity subsolutions $w$ of $G(x, Du)=c(G)$ satisfying
\begin{align*}
	\int_{M} w(x)d\nu(x)\leqslant 0,\quad \forall~\nu\in\mathfrak{M}(L_G).
\end{align*}

Now we are ready to prove the main theorem:
\begin{The}\label{The_uni_sol}
Let $\{H^\lambda\}_{\lambda\in(0,1]}$ be the Hamiltonians satisfying  \mbox{\rm{(SH1)--(SH5)}}. Then   equation \eqref{eq_H_lambda} has a  unique continuous viscosity solution $u^\lambda$,  
which is also Lipschitzian, and the family $\{u^\lambda\}_{\lambda>0}$ converges uniformly, as $\lambda\to 0$, to a single $u^0$ which  is a Lipschitz continuous viscosity solution of \eqref{eq_G}.  Furthermore, the limit $u^0$ is characterized by
\begin{equation}\label{rep_case_G}
	u^0(x)=\sup\limits_{u\in\mathcal{F}(G)} u(x)=\min\limits_{\mu\in \mathfrak{M}(L_G)}\int_{M}h(y,x)\ d\mu(y),
\end{equation}
where  $h(y,x)$ is the Peierls barrier of the Lagrangian $L_G$.
\end{The}

\begin{proof}
By Theorem \ref{uniq_equi_Lip},  we can fix a positive constant $R_0>0 $ so that
   	$$
   	\max\{ \|u^\lambda \|_\infty ,~\text{Lip}(u^\lambda)\} \leqslant R_0, \quad \text{for all~}  \lambda\in(0,1].
   	$$
Now we claim that the functions $$u_-^\lambda:=u^\lambda -\frac{(K^\lambda_{R_0}-\delta^\lambda_{R_0})R_0}{K_{R_0}^\lambda},\quad u_+^\lambda:=u^\lambda +\frac{(K^\lambda_{R_0}-\delta^\lambda_{R_0})R_0}{K_{R_0}^\lambda}$$ are, respectively, a  subsolution and a supersolution of the discounted equation 
\begin{equation}\label{lambda_disc_eq}
	K_{R_0}^\lambda u+G(x,Du)=c(G),\quad x\in M.
\end{equation}
Indeed, this can be easily deduced from assumption (SH5), i.e. for any $|p|_x\leqslant R_0$ and $|u|\leqslant R_0$,
$$K_{R_0}^\lambda u-(K_{R_0}^\lambda-\delta_{R_0}^\lambda)R_0\leqslant H^\lambda(x,p,u)-G(x,p)\leqslant K_{R_0}^\lambda u+(K_{R_0}^\lambda-\delta_{R_0}^\lambda)R_0.$$ 

On the other hand, under our assumptions,  we know from Theorem \ref{DFIZ_result} that   \eqref{lambda_disc_eq} admits a unique continuous viscosity solution,  denoted by $w^\lambda$. By the comparison principle (Theorem \ref{The_comp_prin}) for the discounted equation \eqref{lambda_disc_eq}, we obtain
$$u_-^\lambda \leqslant w^\lambda\leqslant u_+^\lambda.$$
Condition (SH5)  implies that 
$$\lim_{\lambda\to 0}\frac{(K^\lambda_{R_0}-\delta^\lambda_{R_0})}{K_{R_0}^\lambda}=0,$$
this yields $\lim_{\lambda\to 0} (u^\lambda- w^\lambda)=0$. As $K_{R_0}^\lambda\to 0$, by Theorem \ref{DFIZ_result}, $w^\lambda$ converges uniformly to a unique $u^0$, which is a viscosity solution of \eqref{eq_G}, and  formula \eqref{rep_case_G} is valid. 
The proof is complete. 
\end{proof}

An argument completely different from the above applies to obtain a convergence result when 
the constant functions are viscosity subsolutions of \eqref{eq_G}, and then we could remove assumption (SH5) and obtain a simpler representation formula for the limit $u^0$. 
\begin{The}\label{uni_const}
Let $\{H^\lambda\}_\lambda$  satisfy  \mbox{\rm(SH1)--(SH4)} and assume that  the constant functions are viscosity subsolutions of \eqref{eq_G}. Then   equation \eqref{eq_H_lambda} has a unique continuous viscosity solution $u^\lambda$,  which is also Lipschitzian, the collection $\{u^\lambda\}$ converges 
uniformly, as $\lambda\to 0$,  to a  Lipschitz continuous viscosity solution $u^0$ of \eqref{eq_G}, and 
\begin{equation}
	u^0(x)=\min\limits_{y\in\mathcal{A}} h(y,x),
\end{equation}
where $h(y,x)$ and $\mathcal{A}$ are, respectively, the Peierls barrier and the Aubry set of  $L_G$.
\end{The}

\begin{proof}
The existence and uniqueness of a viscosity solution follow from Theorem \ref{uniq_equi_Lip}. Now, we set
	\begin{equation*}
		\hat{u}(x):=\min\limits_{y\in\mathcal{A}} h(y,x).
	\end{equation*} 
Since constant functions are viscosity subsolutions of \eqref{eq_G}, we see by (3) of Proposition \ref{proper_criti_sol} that $h(y,x)\geqslant 0$ for all $x$, $y$. Hence, in view of (2) of Proposition \ref{proper_criti_sol}, we see immediately  
that $\hat{u}\geqslant 0$ in $M$ and it is a viscosity solution of \eqref{eq_G}. Moreover,  it is easy to see that $\hat{u}$ 
and $0$ are, respectively, a  supersolution and a  subsolution of \eqref{eq_H_lambda} because of the monotonicity of $u\mapsto H^\lambda(x,p,u)$ 
and the identity $H^\lambda(x,p,0)=G(x,p)$. Thus, by the comparison principle,
 \begin{equation*}
	0\leqslant u^\lambda\leqslant \hat{u}.
\end{equation*}
In particular, since $\hat{u}$ equals to zero on the Aubry set  $\mathcal{A}$, this implies 
\begin{equation}\label{a0}
	u^\lambda(x)=0,\quad \forall~x\in \mathcal{A}.
\end{equation}

To finish our proof, it suffices, by Proposition \ref{Vanishing}, to prove that any converging subsequence has the same limit $\hat{u}$. Indeed, if $u^0$ is the limit of a convergent sequence 
$\{u^{\lambda_n}\}_{n}$, then, by \eqref{a0}, $u^0(x)=0$ for all $x\in\mathcal{A}$. Furthermore, 
the assertion (4) of Proposition \ref{proper_criti_sol}, a well known property of Aubry sets, 
ensures  that $u^0=\hat u$. 
\end{proof}

\section{Applications: Nonlinear vanishing discount problem}\label{section_app}

In order to verify that our assumptions in Section \ref{section_main} are reasonable and rational,  
we deal with some nonlinear discounted systems and demonstrate 
the convergence results in this section.

\subsection{Application I}
Firstly, 
we consider a  simple example which is a direct  generalization of the discounted equations.  Suppose that 
\begin{enumerate}[\rm(C1)] 
\item   $H\in C(T^*M)$ is convex and coercive in the fibers.
\item  $f(x, u)\in C^1(M\times\R)$ with  $f_u>0$ and 
\begin{equation}\label{loc_lip}
	\limsup_{|u|\to 0}\frac{|f_u(x,u)-f_u(x,0)|}{|u|}<+\infty, \quad \text{uniformly for~}  x\in M.
\end{equation}
\end{enumerate} 
Let $c=c(G)$ be the critical value of the Hamiltonian $$G(x,p):=f(x,0)+H(x,p).$$
Now for each $\lambda\in (0,1]$, we consider the following equation:
\begin{equation}\label{app3_lambda}
f(x,\lambda u)+H(x,Du)=c, \quad x\in M.
\end{equation}
and
 \begin{equation}\label{app3_0}
 f(x,0)+H(x,Du)=c, \quad x\in M.
 \end{equation} 
 
\begin{The}\label{gen_dis_sys_general}
Under the above assumptions \mbox{\rm(C1)--(C2)},  equation \eqref{app3_lambda} has a unique continuous  viscosity solution  $u^\lambda$, which is also Lipschitz continuous, 
and the convergence   
\begin{equation*}
	u^0(x)=\lim_{\lambda\to 0}u^\lambda(x), \quad \text{uniformly for all~} x\in M.
\end{equation*} 
holds, where the limit $u^0$ is a Lipschitz viscosity solution of \eqref{app3_0}. 

The limit  $u^0$ is characterized by
\begin{equation}\label{gen_dis_sys_general_charac}
	u^0(x)=\sup\limits_{u\in\mathcal{F}(\breve{G})} u(x)=\min\limits_{\mu_*\in \mathfrak{M}(L_{\breve{G}})}\int_{M}\breve{h}(y,x)\ d\mu_*(y),
\end{equation}
where  $\breve{h}$ is the Peierls barrier  associated with the new Hamiltonian $\breve{G}(x,p)=\frac{G(x,p)-c}{f_u(x,0)}.$ 
\end{The}

\begin{proof}
We apply Theorem  \ref{The_uni_sol} 
to prove the theorem.  
For each $\lambda>0$, set  $H^\lambda(x,p,u)$ $=f(x,\lambda u)+H(x,p)$. Since  
$f_u(x,0)>0$, one can easily observe that $u^\lambda$ is a viscosity solution of 
\eqref{app3_lambda}
if and only if $u^\lambda$ is a viscosity solution of 
\begin{equation}\label{reduce_gen1}
	\frac{H^\lambda(x,Du,u)-c}{f_u(x,0)}=0.
\end{equation}
Similarly,  $u$ is a viscosity solution of  \eqref{app3_0}
if and only if $u$ is a viscosity solution of the equation 
\begin{equation}\label{reduce_gen2}
\frac{G(x,Du)-c}{f_u(x,0)}=0.
\end{equation}
This leads us to consider a new family of continuous Hamiltonians 
$$
\breve{H}^\lambda(x,p,u)=\frac{H^\lambda(x,p,u)-c}{f_u(x,0)}.
$$
and the corresponding
\begin{equation}\label{newG}
	\breve{G}(x,p)=\frac{G(x,p)-c}{f_u(x,0)}.
\end{equation}
The partial derivative of $\breve{H}^\lambda$ with respect to $u$ is 
\begin{equation*}
	\breve{H}^\lambda_u=\frac{H^\lambda_u}{f_u(x,0)}=\frac{\lambda f_u(x,\lambda u)}{f_u(x,0)}>0.
\end{equation*}
Thus one can easily find that  $\breve{H}^\lambda$ satisfies all  assumptions \mbox{\rm{(SH1)--(SH4)}}. 

It only remains to check  assumption (SH5). For every  $R>0$,  let $$ \delta^\lambda_R :=\min\limits_{\substack{(x,p)\in T^*M\\|u|\leqslant R}} \breve{H}^\lambda_u(x,p,u), \quad K_R^\lambda :=\max\limits_{\substack{(x,p)\in T^*M\\|u|\leqslant R}} \breve{H}^\lambda_u(x,p,u).$$
It suffices to prove that
\begin{equation}
	\lim\limits_{\lambda\to 0}\frac{\delta^\lambda_R}{K^\lambda_R}=1, \quad \text{for every~} R>0.
\end{equation}
Indeed,  from \eqref{loc_lip} we know that there exist $\epsilon>0$ and $B>0$ such that
\begin{equation*}
	|f_u(x,u)-f_u(x,0)|\leqslant B|u|, \quad\text{for all}~ x\in M, |u|\leqslant \epsilon,
\end{equation*}
then as long as $\lambda$ is small enough so that $\lambda R<\epsilon$, we have
\begin{equation*}
	 \lambda\bigg( 1-\frac{\lambda RB}{f_u(x,0)}\bigg)\leqslant\breve{H}^\lambda_u(x,p,u)= \frac{\lambda f_u(x,\lambda u)}{f_u(x,0)}\leqslant \lambda\bigg( 1+\frac{\lambda RB}{f_u(x,0)}\bigg), \quad \forall ~|u|\leqslant R.
\end{equation*}
 
Hence, if we denote $a=\min\limits_{x}f_u(x,0)$, then for $\lambda\ll 1$,
$$\delta^\lambda_R\geqslant\lambda\big(1-\frac{\lambda RB}{a}\big), \quad K^\lambda_R\leqslant \lambda\big(1+\frac{\lambda RB}{a}\big). $$
this leads to
$$
1=\liminf\limits_{\lambda\to 0}\frac{ 1-\frac{\lambda RB}{a}}{ 1+\frac{\lambda RB}{a}}\leqslant \liminf\limits_{\lambda\to 0}\frac{\delta^\lambda_R}{K^\lambda_R}\leqslant \limsup\limits_{\lambda\to 0}\frac{\delta^\lambda_R}{K^\lambda_R}\leqslant 1,
$$
which proves  assumption  (SH5).  
Thus, by Theorem \ref{The_uni_sol},  the   equation $$\breve{H}^\lambda(x,Du,u)=0$$ admits a unique  continuous viscosity solution $u^\lambda$, and $u^\lambda$ uniformly converges, as $\lambda \to 0$, to a unique $u^0$, which is a continuous viscosity solution of $\breve{G}(x,Du)=0$. 
Theorem \ref{The_uni_sol}, moreover, guarantees that  \eqref{gen_dis_sys_general_charac} is valid. 
\end{proof}

Notice that the limit $u^0$ in Theorem \ref{gen_dis_sys_general} is characterized by a new Hamiltonian $\breve{G}$. A natural question is: Can we represent the limit $u^0$ in terms of the original Hamiltonian $G$?
To answer 
this question, 
we analyze the dynamical links   between  $\breve{G}$ and  $G$.  

\begin{Pro}\label{Pro_trans_equi}
Let  $H\in C(T^*M)$ be convex and coercive in the fibers with critical value $c$, and $f\in C^1(M)$ with $f>0$. Set
\begin{align*}
	\breve{H}(x, p):=\frac{H(x, p)-c}{f(x)},\quad x\in M, p\in T^*_xM. 
\end{align*}
Let $L$ and $\breve{L}$ be the Lagrangians associated with $H$ and $\breve{H}$, respectively. 
\begin{enumerate}[\rm(1)]
	\item Let  $\breve{h}$ and $h$ be the Peierls barriers of the Lagrangians $\breve{L}$ and $L$,  respectively. Then 
	\begin{equation}\label{P_M_equ}
		\breve{h}(y, x)=h(y, x).
	\end{equation}
In particular, the projected  Aubry sets  $\mathcal{A}(\breve{L})$ and $\mathcal{A}(L)$ are identical.
 \item There is a one-to-one correspondence between the Mather measures 
 $\widetilde{\mathfrak{M}}(L)$ and $\widetilde{\mathfrak{M}}(\breve{L})$ associated with $L$ and $\breve L$, respectively. More precisely, for any $\tilde{\mu}\in \widetilde{\mathfrak{M}}(L)$, if the Borel measure  
$\tilde{\mu}_*$ is defined  by  
 \begin{equation}\label{meas_1}
 	\int\psi(x, v)\ d\tilde{\mu}_*:= \frac{1}{\int f(x)\ d\tilde{\mu}}\int \psi\bigg(x,\frac{v}{f(x)}\bigg)f(x)\ d\tilde{\mu}, \quad \forall~\psi\in C_c(TM),
 \end{equation} 
then $\tilde{\mu}_*\in \widetilde{\mathfrak{M}}(\breve{L})$. 
Conversely, for any  $\tilde{\mu}_*\in \widetilde{\mathfrak{M}}(\breve{L})$, 
if $\tilde{\mu}$ is defined by 
 \begin{equation}\label{meas_2}
 	\int\psi(x, v)\ d\tilde{\mu}:= \frac{1}{\int 1/f(x)\ d\tilde{\mu}_*}\int\psi(x, f(x)v)/ f(x)\ d\tilde{\mu}_*, \quad \forall~\psi\in C_c(TM).
 \end{equation}
then $\tilde\mu\in \widetilde{\mathfrak{M}}(L)$. 
\end{enumerate}
\end{Pro}

\begin{proof} 
It is immediate to see that if $u\in C(M)$ is a viscosity solution of  $H(x, D u)=c$, then 
it is also a viscosity solution of  $\breve{H}(x, Du)=0$.
Since $c$ is the critical value of $H$, this means that $0$ is the critical value of $\breve{H}$. 

A simple manipulation shows that 
\[
\breve L(x,v)=\frac{L(x,f(x)v)+c}{f(x)}, \ \ \ x\in M,\,v\in T_xM. 
\]

\textbf{(1) } 
Now, we  define the semi-distances 
$S\in C(M\tim M)$ and $\breve S\in C(M\tim M)$, associated with the Hamiltonians $H-c$ and $\breve H$, respectively, by 
\[\begin{aligned}
S(y,x)&\,=\sup\{\psi(x)-\psi(y)\mid \psi\in\cS(H-c)\},
\\ 
\breve S(y,x)
&\,=\sup\{\psi(x)-\psi(y)\mid \psi\in\cS(\breve H)\}.
\end{aligned}
\] 
It is clear that $\cS(H-c)=\cS(\breve H)$, hence 
$S=\breve S$.   Thus, by Proposition \ref{Pro_rep_h}, we see that $h=\breve h$. 
Furthermore, we have $\mathcal{A}(\breve{L})=\{x\in M : \breve{h}(x, x)=0\}=\{x\in M : h(x, x)=0\}=\mathcal{A}(L)$. 		
	
\textbf{(2) } If $\tilde{\mu}\in\widetilde{\mathfrak{M}}(L)$, then the measure $\tilde{\mu}_*$, defined by \eqref{meas_1}, is  a Borel probability measure. Because $\tilde{\mu}$ is closed, then  $\tilde{\mu}_*$ is also a closed measure since
 \begin{equation*}
 	\int \langle D\psi,v\rangle_x\ d\tilde{\mu}_*= \frac{\int \langle D\psi,v\rangle_x\ d\tilde{\mu}}{\int f(x)\ d\tilde{\mu}}=0, \quad \forall~\psi\in C^1(M).
 \end{equation*}
As the critical value of $\breve{L}$ is zero, now it only remains to check that $\int \tilde{L}\ d\tilde{\mu}_*=0$. Indeed, 
  \begin{equation*}
 	\int \tilde{L}(x,v)\ d\tilde{\mu}_*= \frac{\int L(x,v)+c\ d\tilde{\mu}}{\int f(x)\ d\tilde{\mu}}=0,
 \end{equation*}
which yields $\tilde\mu\in \widetilde{\mathfrak{M}}(L)$. Similarly, one can prove that if $\tilde\mu_*\in \widetilde{\mathfrak{M}}(\breve{L})$, then $\tilde{\mu}\in\widetilde{\mathfrak{M}}(L)$ with $\tilde{\mu}$  defined as \eqref{meas_2}. Thus there is a one-to-one correspondence between $\widetilde{\mathfrak{M}}(L)$ and $\widetilde{\mathfrak{M}}(\breve{L})$.
\end{proof}

\begin{The}\label{equivalence}
Under the same assumptions as in Theorem \ref{gen_dis_sys_general}, the limit  $u^0$  in \eqref{gen_dis_sys_general_charac} is also characterized by 	
\begin{align*}
	u^0(x)=\sup\limits_{w\in\mathcal{FR}(G)} w(x)=\min\limits_{\mu\in \mathfrak{M}(L_G)}\frac{\int_{M}f_u(y,0)h(y,x)\ d\mu(y)}{\int_{M}f_u(y,0)\ d\mu(y)},
\end{align*}
where $\mathcal{FR}(G)$ is the set of all viscosity subsolutions $w$ of \eqref{app3_0} satisfying 
\begin{equation*}
	\int_{M} f_u(y,0)w(y)d\mu(y)\leqslant 0,\quad \forall~\mu\in\mathfrak{M}(L_G).
\end{equation*} 
\end{The}
\begin{proof}
If $L_G$ is the Lagrangian associated with $G=f(x,0)+H(x, p)$, then the Lagrangian  associated with $\breve{G}(x,p)=\frac{G(x,p)-c}{f_u(x,0)}$ is
$$\breve{L}_G(x,v)=\frac{L_G(x,f_u(x,0)v)+c}{f_u(x,0)}.$$ 
By  Proposition \ref{Pro_trans_equi}  and \eqref{gen_dis_sys_general_charac} in Theorem \ref{gen_dis_sys_general}, we easily get
\begin{equation*}
	\begin{split}
		\min\limits_{\mu_*\in \mathfrak{M}(\breve{G})}\int_{M}\breve{h}(y,x)\ d\mu_*(y)&=\min\limits_{\mu_*\in \mathfrak{M}(L_{\breve{G}})}\int_{M}h(y,x)\ d\mu_*(y)\\
		&=\min\limits_{\mu\in \mathfrak{M}(L_G)}\frac{\int_{M}f_u(y,0)h(y,x)\ d\mu(y)}{\int_{M}f_u(y,0)\ d\mu(y)}.
	\end{split}
\end{equation*}

On the other hand, $w$ is a viscosity subsolution of  \eqref{app3_0} if and only if $w$ is a viscosity subsolution of  $\breve{G}(x,Du)=0$. Thus, if $w\in\mathcal{F}(\breve{G})$, i.e.  $\int_{M} w(y)d\mu_*(y)\leqslant 0$ for all $\mu_*\in\mathfrak{M}(L_{\breve{G}})$, then by Proposition \ref{Pro_trans_equi} again, 
$$
\frac{\int_{M} f_u(y,0)w(y)d\mu(y)}{\int_{M} f_u(y,0)d\mu(y)}\leqslant 0,\quad \forall~\mu\in \mathfrak{M}(L_G).
$$
Equivalently, we have that
$$
\int_{M} f_u(y,0)w(y)d\mu(y)\leqslant 0,\quad \forall~\mu\in \mathfrak{M}(L_G),
$$
since $\int_{M} f_u(y,0)d\mu(y)>0$. This leads to
\begin{equation*}
u^0=\sup\limits_{w\in\mathcal{F}(\breve{G})} w(x)=\sup\limits_{w\in\mathcal{FR}(G)} w(x). \qedhere
\end{equation*} 
\end{proof}

\subsection{Application II}
Now we study a more general class of systems, let $H=H(x,p,u)$ and $G(x,p):=H(x,p,0)$ be 
the Hamiltonians  satisfying
\begin{enumerate}[\rm(D1)]
\item $H\in C(T^*M\times\R)$, the partial derivative $H_u\in C(T^*M)$ with  $H_u>0$. For any $R>0$, there exists  a constant $B_R>0$ so that
\begin{equation}\label{AppIIC0}
	|H(x,p,u)-H(x,p,0)|\leqslant B_R|u|,\quad
	 \text{for all~} |u|\leqslant R.
    \end{equation}
    and
    \begin{equation}\label{AppIIC1}
    	|H_u(x,p,u)-H_u(x,p,0)|\leqslant B_R|u|, \quad\text{for all~} |p|_x\leqslant R, |u|\leqslant R.
    \end{equation}
  \item $G(x,p)$ and $\frac{G(x, p)-c}{H_u(x,p,0)}$ is convex and coercive in the fibers.   
\end{enumerate}  
Let $c=c(G)$ be  the critical value of $G$, for each $\lambda>0$, we consider the equations
\begin{equation}\label{app4_lambda}
	H(x, Du, \lambda u)=c, \quad x\in M.
\end{equation}
and
 \begin{equation}\label{app4_0}
H(x, Du, 0)=c, \quad x\in M.
 \end{equation} 
For example, we can take $$H(x,p,u)=u+\frac{1}{2} \cos^2u\cdot\sin p+\frac{p^2}{2}+V(x)$$ where $(x,p)\in\T^n\times\R^n$ and $ \delta\ll 1.$ Notice that $\frac{1}{2}\leqslant H_u\leqslant \frac{3}{2}$ and $H_u(x,p,0)=1$.

\begin{The}\label{gen_dis_sys_gg}
Under assumptions \mbox{\rm(D1)--(D2)} above, equation \eqref{app4_lambda} has a unique continuous viscosity solution $u^\lambda$, which is also Lipschitz continuous,  
and the convergence  
\begin{equation*}
	u^0(x)=\lim_{\lambda\to 0}u^\lambda(x), \quad \text{uniformly for all~} x\in M,
\end{equation*} 
holds. Moreover, the limit function  $u_0$ is a Lipschitz viscosity solution of \eqref{app4_0} and it is characterized by  
\begin{equation}\label{gen_dis_sys_gg_charac}
	u^0(x)=\sup\limits_{u\in\mathcal{F}(\breve{G})} u(x)=\min\limits_{\mu_*\in \mathfrak{M}(L_{\breve{G}})}\int_{M}\breve{h}(y,x)\ d\mu_*(y),
\end{equation}
where  $\breve{h}(y,x)$ is the Peierls barrier of the Lagrangian $L_{\breve{G}}$ associated with the Hamiltonian 
$\breve{G}(x,p)$ $=\frac{G(x,p)-c}{H_u(x,p,0)}.$
\end{The}
\begin{proof}
Set $H^\lambda(x,p,u)=H(x,p,\lambda u)$. It follows from hypotheses (D1)--(D2) that 
$H^\lambda$ satisfies assumptions (SH1)-(SH4).
Hence, by Theorem \ref{uniq_equi_Lip}, equation \eqref{app4_lambda} has 
a unique continuous viscosity solution, and the whole family  $\{u^\lambda\}_\lambda$ 
 is equi-bounded and equi-Lipschitz, i.e. 
there exists a constant $R_0>0 $ so that
   	$$
   	\max\{ \|u^\lambda \|_\infty ,~\text{Lip}(u^\lambda)\} \leqslant R_0, \quad \text{for all~}  \lambda\in(0,1].$$

Since $H_u>0$, one can easily observe that $u^\lambda$ is a viscosity solution of  
\eqref{app4_lambda}
if and only if $u^\lambda$ is a viscosity solution of 
\begin{equation}\label{reduce_gg1}
	\frac{H(x, Du, \lambda u)-c}{H_u(x, Du, 0)}=0.
\end{equation}
Similarly,  $u$ is a viscosity solution of equation \eqref{app4_0} 
if and only if $u$ is a viscosity solution of 
\begin{equation}\label{reduce_gg2}
\frac{H(x, Du,0)-c}{H_u(x, Du, 0)}=0.
\end{equation}
This leads us to introduce a  family of new  Hamiltonians 
$$
\breve{H}^\lambda(x, p,u)=\frac{H(x, p, \lambda u)-c}{H_u(x,p, 0)}.
$$
and the corresponding Lagrangians 
$$\breve{G}(x, p)=\frac{G(x, p)-c}{H_u(x,p,0)}.$$
The partial derivative of $\breve{H}^\lambda$ with respect to $u$ is 
\begin{equation*}
	\breve{H}^\lambda_u(x, p, u)=\frac{\lambda H_u(x, p, \lambda u)}{H_u(x,p, 0)}>0.
\end{equation*}

For every number $R>0$,  let $$ \delta_R^\lambda :=\min\limits_{\substack{x\in M\\|u|,|p|_x\leqslant R}} \breve{H}^\lambda_u(x,p, u), \quad K^\lambda_R :=\max\limits_{\substack{x\in M\\|u|,|p|_x\leqslant R}} \breve{H}^\lambda_u(x,p, u).$$
Now we claim that $\breve{H}^\lambda$ satisfy assumption (SH5). It suffices to prove that 
\begin{equation}\label{gg_check}
	\lim\limits_{\lambda\to 0}\frac{\delta^\lambda_R}{K^\lambda_R}=1, \quad \text{for every~} R>0.
\end{equation}
Indeed,  for all $|p|_x, |u|\leqslant R$, by \eqref{AppIIC1} we have
\begin{equation*}
	\lambda\bigg( 1-\frac{\lambda RB_R}{H_u(x,p,0)}\bigg)\leqslant\breve{H}^\lambda_u(x,p,u)= \frac{\lambda H_u(x,p, \lambda u)}{H_u(x,p, 0)}\leqslant\lambda\bigg( 1+\frac{\lambda RB_R}{H_u(x,p,0)}\bigg),
\end{equation*}  
Thus if we denote $a:=\min\limits_{|p|_x\leqslant R}|H_{u}(x,p, 0)|>0$,
$$\delta^\lambda_R\geqslant\lambda\big(1-\frac{\lambda RB_R}{a}\big), \quad K^\lambda_R\leqslant \lambda\big(1+\frac{\lambda RB_R}{a}\big). $$
This leads to
$$
1=\liminf\limits_{\lambda\to 0}\frac{ 1-\frac{\lambda RB_R}{a}}{ 1+\frac{\lambda RB_R}{a}}\leqslant \liminf\limits_{\lambda\to 0}\frac{\delta^\lambda_R}{K^\lambda_R}\leqslant \limsup\limits_{\lambda\to 0}\frac{\delta^\lambda_R}{K^\lambda_R}\leqslant 1.
$$
So \eqref{gg_check} is valid, which implies  assumption (SH5). 

On the other hand, by the same argument in the proof of Theorem \ref{The_uni_sol}, we know that $$u_-^\lambda:=u^\lambda -\frac{(K^\lambda_{R_0}-\delta^\lambda_{R_0})R_0}{K_{R_0}^\lambda},\quad u_+^\lambda:=u^\lambda +\frac{(K^\lambda_{R_0}-\delta^\lambda_{R_0})R_0}{K_{R_0}^\lambda}$$ are, respectively, a subsolution and a supersolution of the discounted equation 
\begin{equation}\label{new_lambda_disc_eq}
	K_{R_0}^\lambda u+\breve{G}(x, Du)=0,\quad \forall~x\in M.
\end{equation}
So by Theorem \ref{DFIZ_result},  equation \eqref{new_lambda_disc_eq} admits a unique continuous viscosity solution,  denoted by $w^\lambda$. By the comparison principle for  equation \eqref{new_lambda_disc_eq}, we obtain
$$u_-^\lambda \leqslant w^\lambda\leqslant u_+^\lambda.$$
Condition (SH5)  implies that 
$$\lim_{\lambda\to 0}\frac{(K^\lambda_{R_0}-\delta^\lambda_{R_0})}{K_{R_0}^\lambda}=0,$$
this yields $\lim_{\lambda\to 0} (u^\lambda- w^\lambda)=0$. As $K_{R_0}^\lambda\to 0$, by Theorem \ref{DFIZ_result}, $w^\lambda$ converges uniformly to a unique $u^0$, which is a viscosity solution of $\breve{G}(x,Du)=0$, and the function  $u^0$ is characterized by \eqref{gen_dis_sys_gg_charac}. This finishes our proof.
\end{proof}

\section{Peierls barriers and Theorem  \ref{DFIZ_result}} \label{sec6}

In this section, we mainly focus on the proof of Theorem \ref{DFIZ_result} and Proposition \ref{Pro_rep_h}. As in Section \ref{section_pre},  we always assume in this section that $M$ is a connected and compact manifold without boundary and $H(x,p):T^*M\to\R$ is a given Hamiltonian that satisfies (H1)--(H3).  
We consider the discounted problem  \addtocounter{equation}{1}
\begin{equation}\label{sec6-1}\tag{DP$_\lambda$}
\lambda u+H(x,Du)=c(H) \ \ \text{ in }\ M,
\end{equation}
and the limit problem
\begin{equation}\label{sec6-2}\tag{DP$_0$}
H(x,Du)=c(H) \ \ \text{ in }\ M.
\end{equation} 

It might be worth recalling (see \erf{L-lower-b},\erf{L-upper-b} and \erf{L-superlinear}) that, under hypotheses (H1)--(H3), the formula \ 
$L(x,v)=\sup_{p\in T_x^*M}[\du{p,v}_x-H(x,p)]$\  
defines an extended real-valued, lower semicontinuous function in $TM$ which is convex and superlinear in each fiber $T_xM$ and bounded in a neighborhood of the zero section of $TM$.

We give some details of the proof of the representation formula \erf{dfiz2016_rep} in 
Theorem \ref{DFIZ_result} in the generality of hypotheses (H1)--(H3).  The argument is a modification of 
that in \cite{DFIZ2016} as suggested or indicated in \cite{AAIY2016}.  
Recall again that the convergence assertion of Theorem \ref{DFIZ_result} has been established 
in \cite{DFIZ2016} under the hypotheses (H1)--(H3).

The following proposition asserts that one of the representation formula \erf{dfiz2016_rep} is valid. 

\begin{Pro} \label{prop6-1} Let $u^0\in \Lip(M)$ be the uniform limit of the family $\{u^\gl\}_{\gl>0}$ of the viscosity 
solutions $u^\gl\in\Lip(M)$ of  \erf{sec6-1}, then 
\beq\label{sec6-12}
u^0(x)=\max\{w(x)\mid w\in\cS(H-c(H)),\ \int_M w(x)\,d\mu(x)\leq 0 \ \ \text{ for all }\ \mu\in\frM(L)\}. 
\eeq
\end{Pro}
 
\def\and{\text{ and }}

\bproof  We show first that   
\beq \label{sec6-3}
\int_M u^0(x)\,d\mu(x)\leq 0 \ \ \text{ for all }\ \mu\in\frM(L).
\eeq
Indeed, by approximation (Proposition \ref{sm_subapprox}), for each $\gl>0$ and $\ep>0$, there 
exists $u^\gl_\ep\in C^1(M)$ such that
\[
H(x,Du^\gl_\ep)\leq c(H) -\gl u^\gl+\ep \ \ \text{ for all } \ x\in M,
\]
which leads to
\[
\du{Du^\gl_\ep,v}_x\leq c(H)+L(x,v)-\gl u^\gl(x) \ \ \text{ for all }\, (x,v)\in TM,
\]
Integration by $\tilde\mu\in\widetilde{\frM}(L)$ yields
\[\bald
0&\,=\int_{TM} \du{Du^\gl_\ep,v}_x \,d\tilde\mu(x,v)\leq \int_{TM}(c(H)+L-\gl u^\gl+\ep)\,d\tilde\mu(x,v)
\\&\,=-\gl\int_M u^\gl(x)\, d\mu(x)+\ep,
\eald\]
where $\mu:=\pi_\#\tilde\mu$, 
which shows, in the limit as $\ep\to 0$ and $\gl\to 0$, that \erf{sec6-3} is valid.

Now the only thing left for us to show is that, for any $w\in\cS\big(H-c(H)\big)$ with $\int_M w\,d\mu(x)\leq 0$ for all $\mu\in\frM(L)$, 
$w\leq u^0$ in $M$. 

Indeed, according to Theorem~\ref{uniq_equi_Lip}, 
we have a constant $R>0$ such that $w$ and the functions $u^\gl$ are Lipschitz continuous with Lipschitz bound $R$. 
Choose any sequence $\{\gd_j\}_{j\in\N}$ 
of positive numbers converging to zero and define $H_R^j$ by 
\[
H_{R}^j(x,p)=H(x,p)+\gd_j\dist(p,B_{x,R})^2,
\]
where $B_{x,R}$ denotes the ball in $T_x^*M$ of radius $R$ with the center at the origin, 
and $\dist(p,B_{x,R})$ denotes the distance in $T^*_xM$ between $p$ and $B_{x,R}$. 
Let $L_R^j$ be the Lagrangian of $H_R^j$. Since $H_R^j$ is superlinear in the fiber $T_x^*M$ for every $x\in M$, the Lagrangians $L_R^j$ are finite valued and continuous in $TM$ and have superlinear growth 
in each fiber $T_xM$.  
Furthermore, $u^\gl$, with $\gl\geq 0$, is a viscosity solution of $\gl u^\gl+H_R^j(x,Du^\gl)=c(H)$ in $M$ 
and $w$ is a viscosity subsolultion of $H_R^j(x,Dw)=c(H)$ in $M$. 

By \cite[Proposition 3.5]{DFIZ2016}, for each $z\in M$, $\gl>0$ and $j\in\N$, there exists $\mu^{z,\gl,j}\in\cP(M)$   
such that
\beq\label{sec6-4}
\gl u^\gl(z)=\int_{TM} \big[L_R^j(x,v)+c(H)\big]\,d\tilde\mu^{z,\gl,j}(x,v),
\eeq
and, moreover, as a direct consequence of  \cite[(3.5)]{DFIZ2016}, 
\beq\label{sec6-5}
\gl\psi(z)=\int_{TM}\big[\du{D\psi,v}_x+\gl\psi(x)\big]\,d\tilde\mu^{z,\gl,j}(x,v)\ \ \text{ for all }\psi\in C^1(M).
\eeq

Note that $u:=w$ is a viscosity subsolution of  $\gl u+H_R^j(x,Du) = c(H)+\gl w$ in $M$,
and, by approximation (Proposition \ref{sm_subapprox}), we may choose $w_\ep$ for each $\ep>0$ so that 
$\|w_\ep-w\|_{\infty}\leq \ep$ and    
\[
\gl w_\ep+\du{Dw_\ep,v}_x \leq c(H)+L_R^j(x,v)+\gl w+\ep \ \ \text{ in }TM.
\]
Integration with respect to $\tilde\mu^{z,\gl,j}$, combined with \erf{sec6-5} and \erf{sec6-4}, yields 
\[\bald
\gl w_\ep(z)&=\int_{TM}[\du{Dv_\ep,v}_x+\gl w_\ep]\,d\tilde\mu^{z,\gl,j}(x,v)
\\&\leq \int_{TM}[c(H)+L_R^j+\gl w+\ep] \,d\tilde\mu^{z,\gl,j}(x,v)  
=\gl u^\gl(z)+\gl \int_M w(x)\,d\mu^{z,\gl,j}(x)+\ep,
\eald\]
where $\mu^{z,\gl,j}:=\pi_\#\tilde\mu^{z,\gl,j}$. 
Hence, by sending $\ep\to 0$ and dividing by $\gl$, we get
\beq\label{sec6-7}
w(z)\leq u^\gl(z)+\int_M w(x)\,d\mu^{z,\gl,j}(x). 
\eeq

Noting that $L_R^j\leq L_R^{j+1}$ and $\lim_{j\to\infty}L_R^j=L$, we find that for all $j, k\in\N$, 
\beq\label{sec6-6}
\int_{TM} \big[c(H)+L_R^j\big]\, d\tilde \mu^{z,\gl,j+k}(x,v) 
\leq \int_{TM} \big[c(H)+L_R^{j+k}\big]\, d\tilde \mu^{z,\gl,j+k}(x,v)= \gl u^\gl(z).
\eeq
Since every Lagrangian $L_R^j$ has superlinear growth in the fibers, \eqref{sec6-6} shows that
the collection of probability measures $\{\tilde \mu^{z,\gl,j}\}_{j\in\N}$ is tight and has a 
weakly convergent subsequence (in the sense of measures), which we denote still by the same symbol, 
to a probability measure $\tilde \mu^{z,\gl}\in\cP(TM)$. 
From \erf{sec6-6}, we get 
\[
\int_{TM} [c(H)+L_R^j]\,d\tilde\mu^{z,\gl}(x,v)\leq \gl u^\gl(z) \ \ \text{ for all }\ j\in\N,
\]
and then by the monotone convergence theorem that 
\[
\int_{TM} [c(H)+L]\,d\tilde \mu^{z,\gl}(x,v)\leq \gl u^\gl(z). 
\]
This shows that the family $\{\mu^{z,\gl}\}_{\gl>0}$ is tight, and we can choose 
a sequence $\{\mu^{z,\gl_k}\}_{k\in\N}$ converging weakly in the sense of measures to 
a $\tilde\mu^{z}\in\cP(TM)$. By the lower semicontinuity of $L$ and the fact that 
$\lim_{\gl \to 0}u^\gl(x)=u^0(x)$ uniformly, we deduce that
\beq \label{sec6-8}
\int_{TM} [c(H)+L]\,d\tilde\mu^{z}(x,v)\leq 0. 
\eeq
Also, it is easily seen from \erf{sec6-5} and \erf{sec6-7} that
\beq \label{sec6-9}
\int_{TM} \du{D\psi,v}_x\,d\tilde \mu^{z}(x,v)=0 \ \ \text{ for }\ \psi\in C^1(M),
\eeq
and 
\beq\label{sec6-10}
w(z)\leq u^0(z)+\int_M w(x)d\mu^{z}(x),
\eeq
where $\mu^z:=\pi_\#\tilde\mu^{z}$. 

The identity \erf{sec6-9} means that $\tilde\mu^z$ is a closed probability measure on $TM$ and, 
this together with \erf{sec6-8} ensures that $\tilde\mu^z$ is a Mather measure for $L$, i.e.
$\tilde\mu^z\in\widetilde\frM(L)$.  Moreover, it follows that 
\[
\int_{TM} [c(H)+L]\,d\tilde\mu^{z}(x,v)=0. 
\]
By the choice of $w$, we have 
$\int_M w(x)\, d\mu^z(x)\leq 0$. Thus, we conclude from \erf{sec6-10} that $w(z)\leq u^0(z)$, where $z\in M$ is arbitrary, and that formula \erf{sec6-12} holds. 
\eproof

The previous and next propositions together validate Theorem \ref{DFIZ_result} concerning 
the representation of the limit $u^0$.  

\begin{Pro} \label{prop6-2}
Assume \emph{(H1)--(H3)}.  Let $u^\gl\in \Lip(M)$ be a unique solution of \erf{sec6-1} for $\gl>0$ 
and  $u^0\in \Lip(M)$ be the uniform limit of $\{u^\gl\}$ as $\gl\to 0$. 
Then
\beq \label{sec6-2-1}
u^0(x)=\min\bigl\{\int_M h(y,x)\, d\mu(y)\mid \mu\in\frM(L)\bigr\} \ \ \text{for}\ x\in M. 
\eeq
\end{Pro}

It should be remarked that Proposition \ref{prop6-2} above guarantees the existence of  Mather 
measures for $L$; otherwise the right side of \erf{sec6-2-1} would equal $+\infty$.  
The following proof parallels that of \cite[Theorem 4.3]{DFIZ2016}, with basis on the 
formula $h(y,x)=\min\limits_{z\in\cA}[S(z,x)+S(y,z)]$ in Proposition \ref{Pro_rep_h}. 

\bproof  We write $w(x)$ for the right side of \erf{sec6-2-1}. First, we show that $u^0\leq w$. Indeed, by (3) of Proposition \ref{proper_criti_sol}, 
\[
u^0(x)-u^0(y)\leq h(y,x) \ \ \text{ for all }x,y\in M.
\] 
Integration in $y$ with respect to $\mu\in\frM(L)$, together with Proposition \ref{prop6-1}, yields 
\[
u^0(x)\leq \int_M u^0(y)\,d\mu(y) +\int_M h(y,x)\, d\mu(y)\leq \int_M h(y,x)\,d\mu(y).
\]
This assures that $u^0\leq w$. 

Next, by (2) of Proposition \ref{proper_criti_sol}, the function 
$y\mapsto -h(y,x)+w(x)$ is a viscosity subsolution of \erf{sec6-2}. Integrating this function with respect to $\mu\in\frM(L)$, 
\[
\int_M [ -h(y,x)+w(x)]\,d\mu(y)=-\int_M h(y,x)\,d\mu(y)+w(x)\leq 0 \ \ \text{ for any } \ x\in M.
\]
The characterization of $u^0$ in Proposition \ref{prop6-1} guarantees that 
\begin{equation}\label{sec6-13}
	u^0(y)\geq -h(y,x)+w(x) \ \ \text{ for all }x, y\in M,
\end{equation}
this yields $u^0(z)\geq w(z)$  for all $z\in\cA.$
Since $u^0$ is a viscosity solution  and $w\in\cS\big(H-c(H)\big)$, it follows from (4) of Proposition \ref{proper_criti_sol} that $w\leq u^0$. 
Thus,  $u^0=w$ in $M$
\eproof 

Now, we start to prove Proposition \ref{Pro_rep_h}, which is a critical issue 
in this section. 

The following proposition is fundamental connecting the functions $h_t$ and the solutions 
of the Hamilton-Jacobi equation $\pl_t u+H(x,Du)=c(H)$, a proof of which is given in the appendix. 

\begin{Pro} \label{Hopf}Let $u_0\in\Lip(M)$, and set 
\beq\label{Hopf1}
U(x,t)=\inf_{y\in M} [u_0(y)+h_t(y,x)] \ \ \text{ for } (x,t)\in M\tim(0,\,\infty).
\eeq
Then, 
\begin{enumerate}[\rm(1)]
\item \quad $\lim_{t\to 0+}U(x,t)=u_0(x) \ \ \text{ uniformly for }\ x\in M,$\smallskip
\item \quad The function $U$ is bounded and Lipschitz continuous in $M\tim (0,\,\infty),$ \smallskip
\item  \quad The function $U$ is  
a viscosity solution of 
\beq\label{eHJ}
\pl_t u+H(x,Du)=c(H) \ \ \text{ in }\ M\tim(0,\,\infty).
\eeq
\end{enumerate}
\end{Pro}

\begin{Pro} \label{comp-cp} Let $0<T\leq \infty$ and $v,w : M\tim[0,\,T) \to \R$ be an upper semicontinuous 
viscosity subsolution and a lower semicontinuous viscosity supersolution of \eqref{eHJ} 
in $M\tim(0,\,T)$. Assume that 
$v(x,0)\leq w(x,0)$ for $x\in M$. Then, $v\leq w$ in $M\tim [0,\,T)$. 
\end{Pro}

Semicontinuous viscosity sub- and super-solutions are needed in our proof of Proposition \ref{Hopf} 
in the appendix. By definition, an upper semicontinuous function $v\mid V\to \R$, where 
$V$ is an open subset of $M\tim(0,\,\infty)$, is a viscosity subsolution of 
$\pl_t v+H(x,Dv)=c(H)$ in $V$ if, for any $\gf\in C^1(V)$ and $(\bar x,\bar t)\in V$ such that 
$(v-\gf)(x,t)\leq (v-\gf)(\bar x,\bar t)$ for $(x,t)\in V$, we have \ $\pl_t \gf(\bar x,\bar t)+H(\bar x, D\gf(\bar x,\bar t))\leq c(H)$. Similarly, a lower semicontinuous function $v\mid V\to \R$, where 
$V$ is an open subset of $M\tim(0,\,\infty)$, is a viscosity subsolution of 
$\pl_t v+H(x,Dv)=c(H)$ in $V$ if, for any $\gf\in C^1(V)$ and $(\bar x,\bar t)\in V$ such that 
$(v-\gf)(x,t)\geq (v-\gf)(\bar x,\bar t)$ for $(x,t)\in V$, we have \ $\pl_t \gf(\bar x,\bar t)+H(\bar x, D\gf(\bar x,\bar t))\geq c(H)$.   

The proof of Proposition \ref{comp-cp} is similar to that of Theorem \ref{The_comp_prin} once one knows the existence of a Lipschitz continuous, viscosity solution of \erf{eHJ} for each Lipschitz 
initial data, as stated in the next proposition. We give main ideas of the proof of Propositions \ref{comp-cp} and \ref{exist-cp} in Appendix B.    

\begin{Pro} \label{exist-cp} For any $u_0\in\Lip(M)$, there exists a viscosity solution $u\in\Lip(M\tim[0,\,\infty))$ of \erf{eHJ} satisfying the initial condition $u(\cdot, 0)=u_0$. 
\end{Pro}

The next lemma is a simple adaptation of Proposition \ref{sm_subapprox} and it is left 
to the reader to prove it. 

\begin{Lem} \label{approx-cp} Let $u\in\Lip(M\tim[0,\,\infty))$ be a viscosity solution  of \erf{eHJ}. 
Then, for each $\ep>0$, there exists $u_\ep\in C^\infty(M\tim[0,\,\infty))$ such that 
\[
\pl_t u_\ep+H(x,Du_\ep)\leq c(H)+\ep \ \ \text{ in }M\tim(0,\,\infty)\quad\and 
\quad \|u-u_\ep\|_\infty<\ep.
\]
\end{Lem}

We need the dynamic programming principle, stated as

\begin{Lem} \label{dpp} Let $\tau>0,\,\sigma>0$  and set $t=\tau+\gs$. Then
\[\bald
h_t(x,y)&\,=\inf_{z\in M}[h_\tau(x,z)+h_\gs(z,y)].
\eald\]
\end{Lem}

Remark that both sides of the formula  in the lemma above can be $+\infty$. This lemma can be proved easily by the definition of $h_t$.

\begin{Pro} \label{S_&_h_t}For any $x,y\in M$, we have
\[
S(x,y)=\inf_{t>0}h_t(x,y).
\]
\end{Pro}

A standard proof of the proposition above is the one similar to that of Proposition \ref{Hopf} 
presented in the Appendix below and based on the dynamic programming principle (Lemma \ref{dpp}). 
The following proof is more dependent on Proposition \ref{Hopf}.

\bproof  Fix any $y\in M$ and set
\[
u(x,t)=\inf_{z\in M}[S(y,z)+h_t(z,x)].
\]
According to Proposition \ref{Hopf}, the function $u\in \Lip(M\tim(0,\,\infty))$ is a viscosity solution of 
\beq\label{cp2}
\pl_t u+H(x,Du)=c(H) \ \ \text{ in  }M\tim(0,\,\infty)
\eeq
and satisfies
\beq\label{cp3}
\lim_{t\to 0}u(x,t)=S(y,x) \ \ \text{ uniformly for }x\in M. 
\eeq
Since the function $(x,t)\mapsto S(y,x)$ is also a viscosity subsolution of \erf{cp2} with \erf{cp3}, it follows from Proposition \ref{comp-cp} that $S(y,x)\leq u(x,t)$, which implies that 
\[ 
S(y,x)\leq h_t(y,x)\ \ \text{ for }t>0.
\]
If we set \ $h^-(x,y)=\inf_{t>0} h_t(x,y)$,
 the inequality above reads 
\beq\label{S<=h^-}
S(y,x)\leq h^-(y,x)\ \ \text{ for }x,y\in M.
\eeq

Let $\gd$ and $C_0$ be the constants in \eqref{L-upper-b}. We choose a constant $r>0$  so that for any $x,y\in M$, if $d(x,y)<r$, then  
there is a geodesic curve $\gamma$ with speed $\gd $ connecting $x$ and $y$ and 
\[
h_{d(x,y)/\gd}(x,y)\leq \int_0^{d(x,y)/\gd}L(\gamma(s),\dot\gamma(s))+c(H)\, ds
\leq (C_0+c(H))d(x,y)/\gd.
\]
Accordingly, since $M$ is compact and connected,  for some constant $C>0$, we have 
\[
h^-(x,y)\leq Cd(x,y) \ \ \text{ for all }x,y\in M. 
\]
We may assume that $C$ is large enough so that $S$ is Lipschitz continuous with Lipschitz bound $C$. 
By \erf{S<=h^-}, we have
\[
h^-(x,y)\geq S(x,y)\geq -Cd(x,y). 
\]
It follows from Lemma \ref{dpp} that 
\[
h^-(x,y)\leq h^-(x,z)+h^-(z,y) \ \ \text{ for all }x,y,z\in M.
\]
These show that $h^-(x,x)=0$ for $x\in M$ and $h^-\in \Lip(M\tim M)$. 

By Lemma \ref{dpp}, 
\beq\label{h_t+h^-}
\inf_{s>0}h_{t+s}(y,x)=\inf_{s>0}\inf_{z\in M}[h_s(y,z)+h_t(z,x)]
=\inf_{z\in M}[h^-(y,z)+h_t(z,x)].
\eeq
Hence, fixing $y\in M$ and setting 
\[
v(x,t)=\inf_{z\in M}[h^-(y,z)+h_t(z,x)],
\]
we observe by Proposition \ref{Hopf} that $v$ is a viscosity solution of \erf{cp2} and satisfies 
$\lim_{t\to 0}v(x,t)=h^-(y,x)$ uniformly for $x\in M$. By \erf{h_t+h^-}, we have  
$v(x,t)=\inf_{s>0}h_{t+s}(y,x)$, which shows that the function $t\mapsto v(x,t)$ is nondecreasing. Hence
we see that, for any $t>0$, $x\mapsto v(x,t)$ is a viscosity subsolution of $H(x,Du)=c(H)$ in $M$ and, 
since $h^-(y,x)=\lim_{t\to 0}v(x,t)$ uniformly, the function  
$x\mapsto h^-(y,x)$ is a viscosity subsolution of $H(x,Du)=c(H)$ in $M$. By the definition of $S$,
we have 
\[
h^-(y,x)=h^-(y,x)-h^-(y,y)\leq S(y,x).
\]   
This and \erf{S<=h^-} yield that $h^-(y,x)=S(y,x)$. 
\eproof

We set temporarily 
\[
\cA_S=\{z\in M\mid \text{$x\mapsto S(z,x)$ is a viscosity solution of }H(x,Du)=c(H)\text{ in }M\}.
\]
Then we have an equivalent description for the projected Aubry set.

\begin{The}\label{equ_Aub}
The sets	$\cA_S=\cA$. 
\end{The}

\bproof
 Fix any $z\in \cA_S$, namely the 
function $x\mapsto S(z,x)$ is a viscosity solution of $H(x,Du)=c(H)$ in $M$. 
As in the proof of Proposition \ref{S_&_h_t}, we set 
\beq \label{def-u}
u(x,t)=\inf_{y\in M}[S(z,y)+h_t(y,x)] \ \ \text{ for }\ (x,t)\in M\tim(0,\,\infty).
\eeq
Observe that the functions $u(x,t)$ and $w(x,t):=S(z,x)$ are both 
viscosity solutions of \erf{eHJ} with the initial condition 
$\lim_{t\to 0}u(x,t)=\lim_{t\to 0}w(x,t)=S(z,x)$ uniformly for $x\in M$. Proposition \ref{comp-cp}
then guarantees that $u=w$. That means \[
\inf_{y\in M}[S(z,y)+h_t(y,x)] =S(z,x). 
\]
This combined with Propositions \ref{S_&_h_t} and \ref{dpp} reveals 
\[
S(z,x)=\inf_{y\in M}[h_t(y,x)+\inf_{s>0}h_s(z,y)]=\inf_{s>0}h_{t+s}(z,x),
\]
and, evaluated at $x=z$,
\[
0=\inf_{s>0}h_{t+s}(z,z)\ \ \text{and}\ \ h(z,z)=\lim_{t\to+\infty}\inf_{s>0}h_{t+s}(z,z)=0. 
\]
Thus, we see that $z\in \cA$. 

Now, let $z\in\cA$, and define $u\mid M\tim(0,\,\infty)\to\R$ by \erf{def-u} as before. 
By Proposition \ref{S_&_h_t}, $u(x,t)\geq\inf_{y\in M}[S(z,y)+S(y,x)]\geq S(z,x)$ for $(x,t)\in M\tim(0,\,\infty)$. 
Next, we observe that $\lim_{t\to 0}u(x,t)=S(z,x)$ uniformly for $x\in M$, 
and 
\[
u(x,t)=\inf_{s>0}h_{t+s}(z,x),
\]
the last of which implies that the function $t\mapsto u(x,t)$ is nondecreasing. 
Hence, 
\[
0=h(z,z)=\lim_{t\to \infty} u(z,t)\geq u(z,t)\geq \lim_{t\to 0}u(z,t)=S(z,z)=0,\quad t>0
\]
This  yields 
\[
u(z,t)=0 \ \ \ \text{ for all }t>0.
\]
The monotonicity implies as well that, for each $t>0$, $v(x):=u(x,t)$ is a viscosity 
subsolution of $H(x,Dv)=c(H)$ in $M$. By the definition of $S$, we have
$u(x,t)=v(x)-v(z)\leq S(z,x)$ for all $x\in M$ and $t>0$, so 
$u(x,t)=S(z,x)$ for $(x,t)\in M\tim(0,\,\infty)$. This shows that $(x,t)\mapsto S(z,x)$ is 
a viscosity solution of \erf{eHJ}, which implies that $w\mid x\mapsto S(z,x)$ is a 
viscosity solution of  $H(x,Dw)=c(H)$ in $M$. Hence, $z\in\cA_S$, finishing the proof.   
\eproof 

Theorem \ref{equ_Aub} has proven the first part of Proposition \ref{Pro_rep_h}, while the second part is as follows:
\begin{The} \label{formula_h} For any $x,y\in M$,  
\[
h(x,y)=\inf_{z\in\cA}[S(x,z)+S(z,y)].
\]
\end{The}
We divide the proof of Theorem \ref{formula_h} into proving the following three lemmas. 

\begin{Lem} \label{h<=S+S}
For any $ x,y\in M$,
\[
h(x,y)\leq \inf_{z\in\cA}[S(x,z)+S(z,y)].
\]
\end{Lem}

\bproof  
Let $\tau, \gs, \gth\in(0,\,\infty)$ and set $t=\tau +\gs+\gth$.  
By Lemma \ref{dpp}, we have
\[\bald
h_t(x,y)&\,=\inf_{z_1,z_2\in M}[h_\tau(x,z_1)+h_\gs(z_1,z_2)+h_\gth(z_2,y)]
\\&\,\leq h_\tau(x,z)+h_\gs(z,z)+h_\gth(z,y) \ \ \text{ for any }\, z\in\cA.
\eald\]
Noting that $h(z,z)=0$ for $z\in \cA$, we take 
the liminf of the both sides as $\gs\to\infty$, to obtain
\[
h(x,y)\leq h_\tau(x,z)+h_\gth(z,y) \ \ \text{ for all }z\in\cA.
\] 
This and Proposition \ref{S_&_h_t} yield
\[
h(x,y)\leq S(x,z)+S(z,y) \ \ \text{ for }z\in\cA,
\]
which completes the proof. 
\eproof

The following lemma is a basic observation in weak KAM theory. 
For the proof, see \cite[Proposition 8.5.3]{Fathi_book}, where a more refined version of the lemma is discussed. See also \cite[Lemma 8.4]{Ishii2008} for details on the proof.

\begin{Lem} \label{Aubry_prop}Let $K\subset M$ be a compact set such that $K\cap \cA=\emptyset$. 
Then there exists a function $\psi\in \Lip(M)$ and $f\in C(M)$ 
such that $H(x,D\psi)\leq c(H)+f$ a.e. in $M$ and $\max_Kf<0$. 
\end{Lem}

\begin{Lem} \label{h>=S+S} We have 
\[
h(x,y)\geq\inf_{z\in\cA} [S(x,z)+S(z,y)]
.\]
\end{Lem}

\bproof 
Let $\ep\in(0,\,1)$ and set  $U=\{x\in M\mid \dist(x,\cA)<\ep\}$, where $\dist(x,\cA)$ denotes the distance of  a point $x$ and the set $\cA$ induced by the metric $d$. By Lemma \ref{Aubry_prop}, 
there exist functions $\psi\in\Lip(M), \,f\in C(M)$ such that $H(x,D\psi)\leq c(H)+f$ a.e. in $M$ and $\max_{M\stm U}f<0$. 
Fix $\gd>0$ so that $\max_{M\stm U}f <-2\gd$. By approximation of $\psi$, we may select  
$\psi_\gd\in C^1(M)$ such that $H(x,D\psi_\gd)\leq c(H)+f+\gd$ in $M$, so that $H(x,D\psi_\gd)\leq c(H)-\gd$ in 
$M\stm U$. 

Let $C>0$ be a constant such that $|S(x,y)|\leq C$ for all  $x,y\in M$ 
and $|\psi_\gd(x)|\leq C$ for all $x\in M$.
By Proposition \ref{S_&_h_t} and Lemma 
\ref{h<=S+S}, we have 
\[
h(x,y)\leq 2C \ \ \text{ for all }x,y\in M. 
\]

Let $T>0$ be such that 
\[
h_T(x,y)<h(x,y)+\ep,
\]
and select  $\gamma\in\Gamma_{x,y}^T$ so that 
\[
\int_0^T [L(\gamma(s),\dot\gamma(s))+c(H)]\,ds<h(x,y)+\ep.
\]
Observe then that if $\gamma(s)\in M\stm U$ for all $s\in [0,\,T]$, then 
\[\bald
\psi_\gd(x)-\psi_\gd(y)&\,= \int_0^T\du{D\psi_\gd(\gamma(s)),\dot\gamma(s)}_{\gamma(s)}\, ds 
\\&\,\leq \int_0^T[L(\gamma(s),\dot\gamma(s))+H(\gamma(s),D\psi_\delta(\gamma(s))]\,ds
\\&\,\leq \int_0^T [L(\gamma(s), \dot\gamma(s))+c(H)-\gd ]\,ds  
<h(x,y)+\ep-\gd T,
\eald\]
and hence,
\[
\gd T<4C+1.
\]
Conversely, if $\gd T\geq 4C+1$, then such curves $\gamma$ as above must 
 intersect the set $U$. 

By the argument above, where $\ep\in(0,\,1)$ is arbitrarily chosen, we deduce that there exist 
sequences $\{t_j\}_{j\in\N}$ and $\{\ep_j\}_{j\in\N}$ of positive numbers 
and, for each $j\in\N$, a curve 
$\gamma_j\in\Gamma_{x,y}^{t_j}$ such that $\lim_{j\to\infty}\ep_j=0$,  $\lim_{j\to\infty}t_j=+\infty$, and
\[
\lim_{j\to \infty}\int_0^{t_j}L(\gamma_j(s),\dot\gamma_j(s))\,ds=h(x,y)\quad \and\quad 
\dist(\gamma_j([0,\,t_j]),\,\cA)<\ep_j,
\] 
where $\dist(A,B)$ denotes the distance of two sets $A,B$ induced by the metric $d$.
We choose $\tau_j\in (0,\,t_j)$ and $z_j\in\cA$ so that $d(\gamma_j(\tau_j),z_j)<\ep_j$ and compute by using Lemma \ref{S_&_h_t} that
\[\bald
\int_0^{t_j}L(\gamma_j(s),\dot\gamma_j(s))\,ds&\,=
\int_0^{\tau_j} L(\gamma_j(s),\dot\gamma_j(s))\,ds+\int_{\tau_j}^{t_j}L(\gamma_j(s),\dot\gamma_j(s))\,ds
\\&\,\geq S(x,\gamma_j(\tau_j))+S(\gamma_j(\tau_j),y)
\\&\,\geq S(x,z_j)+S(z_j,y)-S(\gamma_j(\tau_j),z_j)-S(z_j,\gamma_j(\tau_j))
\\&\,\geq \inf_{z\in\cA}[S(x,z)+S(z,y)]+O(\ep_j)
\eald\]
as $j\to \infty$. Thus, we have $h(x,y)\geq \inf_{z\in\cA}[S(x,z)+S(z,y)]$ and finish the proof. 
\eproof
\bproof[Proof of Theorem \ref{formula_h}] We only need to combine Lemmas \ref{h<=S+S} and \ref{h>=S+S}, to finish the proof.
\eproof 

\appendix

\section{} 
In this appendix, we give  a proof of Proposition \ref{Hopf}, which follows mostly that of  \cite[Theorem 5.1]{Ishii2011}.

We need to use a version of \cite[Lemma 5.5]{Ishii2011} stated as follows:

\begin{Lem} \label{Lemma 5.5} Let $\gf\in C^1(M\tim[0,\,\infty))$, $(\bar x,\bar t)\in M\tim(0,\infty)$,
and $\ep\in(0,\,1)$. Then, there exists an absolutely continuous curve $\gamma
\mid[0,\,\bar t]\to M$ such that $\gamma(\bar t)=\bar x$ and  
\[
L\big(\gamma(s),\dot\gamma(s)\big)+H\big(\gamma(s),D\gf(\gamma(s),s)\big)\leq \ep +\du{D\gf(\gamma(s),s),\dot\gamma(s)}_{\gamma(s)} \ \ \text{ for a.e. }s\in[0,\,\bar t]. 
\] 
Furthermore, there exists a constant $C>0$, depending only on $\|D\gf\|_{C(M\tim[0,\,\bar t])}$ and $H$, such that 
$|\dot\gamma(s)|_{\gamma(s)}\leq C$\ for a.e. $s\in[0,\,\bar t]$. 
\end{Lem}

The proof of \cite[Lemmas 5.5 and 5.6]{Ishii2011} can be easily modified when one works in a local chart, 
with the Euclidean inner product replaced by those given by the Riemannian metric of $M$ 
and with interpretation of the curve $\gamma$ above  
as the map $s\mapsto (\gamma(\bar t-s),-\dot \gamma(\bar t-s),0)$ being a member of $\text{SP}(x)$, where $\text{SP}(x)$ is defined as the collection of solutions of the Skorokhod problem (see \cite{Ishii2011}). 
We leave to the reader to check the validity of the lemma above.

\bproof [Proof of Proposition \ref{Hopf}] Fix any $u_0\in\Lip(M)$ and let $U$ and $V$ be 
the function given by  \erf{Hopf1} and the function $u$ in Proposition \ref{exist-cp}, respectively. 

We first prove that 
\beq\label{bound_from_below}
V(x,t)\leq U(x,t) \ \ \text{ for }(x,t)\in M\tim (0,\,\infty).
\eeq
For this, fix any $\ep>0$ and, in view of Lemma \ref{approx-cp}, choose 
$V_\ep\in C^1(M\tim[0,\,\infty))$ such that 
\[
\pl_t V_\ep+H(x,DV_\ep)\leq c(H)+\ep \ \ \text{ in }M\tim(0,\infty)
\quad\and\quad V_\ep-\ep\leq V\leq V_\ep+\ep \ \ \text{ in }M\tim[0,\,\infty).
\]
Let $t>0$, $x,y\in M$, and $\gamma\in\Gamma^t_{y,x}$. Compute that
\[\bald
V_\ep(x,t)-V_\ep(y,0)&\,=V_\ep(\gamma(t),t)-V_\ep(\gamma(0),0)
\\&\,=\int_0^t[\du{DV_\ep(\gamma(s),s),\dot\gamma(s)}_{\gamma(s)} +\pl_t V_\ep(\gamma(s),s)]ds
\\&\,\leq \int_0^t [H(\gamma(s),DV_\ep(\gamma(s))+L(\gamma(s),\dot\gamma(s)) +\pl_t V_\ep(\gamma(s),s)]ds
\\&\,\leq \ep t+\int_0^t [L(\gamma(s),\dot\gamma(s))+c(H)]\,ds,
\eald\]
and hence
\[
V(x,t)\leq u_0(y)+2\ep +\ep t+\int_0^t [L(\gamma(s),\dot\gamma(s))+c(H)]\,ds. 
\]
As $\gamma\in\Gamma^t_{y,x}$ and $\ep>0$ are arbitrary, we get from the above
\[
V(x,t)\leq u_0(y)+h_t(y,x),  
\]
which implies furthermore that inequality \erf{bound_from_below} holds. 

Next, we show that there is a constant $C_0>0$ such that
\beq\label{C_0_bound}
U(x,t)\leq u_0(x)+C_0 t \ \ \text{ for }(x,t)\in M\tim(0,\,\infty). 
\eeq
To see this, we choose $C_0>0$ so that 
\[
c(H)-\min_{T^*M}H(x,p)\leq C_0,
\]
and observe by the convex duality that 
\[
L(x,0)+c(H)=-\min_{p\in T_x^*M}H(x,p)+c(H)\leq C_0 \ \ \text{ for all }\ x\in M,
\]
and, for any $(x,t)\in M\tim(0,\,\infty)$, the curve $\gamma_x(s)\equiv x$ belongs to $\Gamma^t_{x,x}$ 
and  
\[
h_t(x,x)\leq \int_0^t [L(\gamma_x(s),0)+c(H)]\,ds\leq C_0 t,
\]
which assures that
\[
U(x,t)\leq h_t(x,x)+u_0(x)\leq u_0(x)+C_0t,
\]
which is exactly inequality \erf{C_0_bound}.

To show the reverse inequality to \erf{bound_from_below}, we only need to show that the upper semicontinuous envelope $U^*$ is a viscosity subsolution of \erf{eHJ}. 
Note that, by definition, we have
\[
U^*(x,t)=\lim_{r\to 0+}\sup\{U(y,s)\mid (y,s)\in B_r(x,t)\},
\]
where $B_r(x,t)$ denotes the ball of radius $r$ with center at $(x,t)$ with respect to the distance 
induced in the Riemannian manifold $M\tim \R$.   By \erf{bound_from_below} and \erf{C_0_bound}, we get
\[
V(x,t)\leq U^*(x,t)\leq u_0(x)+C_0 t\ \ \text{ for } \ (x,t)\in M\tim[0,\,\infty),
\]
which, in particular, shows that $V(x,0)=U^*(x,0)=u_0(x)$ for $x\in M$. 
Once these observations are done, by Proposition \ref{comp-cp}, we get $U^*\leq V$ in $M\tim[0,\,\infty)$, 
which shows that $U=V$ in $M\tim(0,\,\infty)$, and the proof ends. 

Now, the only thing left for us to show is that, the upper semicontinuous envelope $U^*$ is a viscosity subsolution of \erf{eHJ}. Let $\gf\in C^1(M\tim [0,\,\infty))$ and assume that $U^*-\gf$ has a strict maximum at 
$(\hat x,\hat t)\in Q:=M\tim(0,\,\infty)$ and $(U^*-\gf)(\hat x,\hat t)=0$.

We need to prove the inequality
\beq 
\pl_t\gf(\hat x, \hat t)+H(\hat x,D\gf(\hat x,\hat t))\leq c(H).
\eeq
We argue by contradiction, and thus suppose 
\[
\pl_t\gf(\hat x, \hat t)+H(\hat x,D\gf(\hat x,\hat t))> c(H).
\]
By continuity, we may choose constants $r>0$ with $\hat t-r>0$, and $\ep\in(0,\,1)$ so that.
\beq \label{sub-ineq}
\pl_t\gf(x, t)+H(x,D\gf(x,t))>c(H)+\ep \ \ \text{ for all }(x,t)\in\ol B_r(\hat x)\tim[\hat t-r,\,\hat t+r],
\eeq
where $B_r(x)$ denotes the geodesic ball of radius $r$ and center $x$. 

Next, we apply Lemma \ref{Lemma 5.5}, in order to find an appropriate curve $\gamma$. First of all, let $C>0$ be the 
constant from Lemma \ref{Lemma 5.5} depending only on $H$ and $\|D\gf\|_{C(M\tim[0,\,\hat t+r])}$.

We choose $\rho\in (0,\,r)$ so that $4C\rho\leq r$. 
Since the maximum value of $U^*-\gf$ is zero and it is a strict maximum, we set 
\[
\gd:=-\max_{\pl(\ol B_r\tim [\hat t-\rho,\hat t+\rho])}(U^*-\gf)>0.
\]
 We may select a point $(x_0,t_0)\in \ol B_{r/2}(\hat x)\tim[\hat t-\rho/2,\,\hat t+\rho/2]$ so that 
\[
(U-\gf)(x_0,t_0)>-\gd.
\]

We invoke Lemma \ref{Lemma 5.5}, to obtain $\gamma$ with $\gamma(t_0)=x_0$ such that 
for a.e.  $s\in[0,\,t_0]$,
\beq\label{l-semi}\left\{ \bald
&H(\gamma(s), D\gf(\gamma(s),s))+L(\gamma(s),\dot\gamma(s))<\ep +\du{D\gf(\gamma(s), s),\dot\gamma(s)}_{\gamma(s)},
\\&|\dot\gamma(s)|_{\gamma(s)}\leq C.
\eald
\right.\eeq

Now, by our choice of $(x_0,t_0)$, setting $\gs=\hat t-\rho$, we have
\[
d(\gamma(s),x_0)\leq \int_s^{t_0}|\dot\gamma(\tau)|_{\gamma(\tau)}\,d\tau
\leq C(t_0-s)<2C\rho\leq \fr r 2 \ \ \text{ for } s\in[\gs,\,t_0],
\]
which shows that $\gamma(s)\in B_r(\hat x)$ for all $s\in[\gs,\,t_0]$, and also,  
\[
(U-\gf)(x_0,t_0)>\max_{\pl(\ol B_r(\hat x)\tim [\gs,\,\hat t+\rho])}(U^*-\gf)
\geq (U-\gf)(\gamma(\gs),\gs).
\]
Hence,
\[\bald
U(x_0,t_0)-U(\gamma(\gs),\gs)
&\,>\gf(x_0,t_0)-\gf(\gamma(\gs),\gs)
\\&\,=\int_\gs^{t_0}\big[\du{D\gf(\gamma(s),s),\dot\gamma(s)}_{\gamma(s)}+\pl_t\gf(\gamma(s),s)\big]\,ds,
\eald\]
and, moreover, using \erf{l-semi} and \erf{sub-ineq},
\[\bald
U(x_0,t_0)&-U(\gamma(\gs),\gs)
\\ &\,>\int_\gs^{t_0} \big[-\ep+L(\gamma(s),\dot\gamma(s))
+H(\gamma(s),D\gf(\gamma(s),s))+\pl_t\gf (\gamma(s),s)\big]\,ds
\\&\,>\int_\gs^{t_0} [L(\gamma(s),\dot\gamma(s))+c(H)]\,ds
=\int_0^{t_0-\gs} [L(\gamma(s+\gs),\dot\gamma(s+\gs))+c(H)]\,ds
\\&\,\geq h_{t_0-\gs}(\gamma(\gs),x_0).
\eald
\] 
Thus, we obtain
\[
U(x_0,t_0)>\inf_{y\in M}\big[U(y,\gs)+h_{t_0-\gs}(y,x_0)\big],
\]
which is a contradiction in view of Lemma \ref{dpp}. This completes the proof.
\eproof 

\section{}

Here we give some ideas of how to prove Propositions \ref{comp-cp} and \ref{exist-cp}. 

\bproof[Outline of proof of Proposition \ref{exist-cp}] For $R>0$, we define the function $\Gth_R\mid\R\to\R$ by $\Gth_R(r)=\min\{r,\,R\}$. Fix $R>0$, set $H_R=\Gth_R\circ (H-c(H))$, and observe that $H_R$ is bounded 
and uniformly continuous in $T^*M$. To establish the existence of a viscosity solution $u\in\Lip(M\tim[0,\,\infty))$ of \erf{eHJ}, we consider 
\beq \label{eHJR}
\pl_t u+H_R(x,Du)=0 \ \ \text{ in }M\tim(0,\infty).
\eeq
Given a function $u_0\in\Lip(M)$, it is important to obtain a viscosity solution $u^R$ of \erf{eHJR} 
satisfying $u(\cdot,0)=u_0$ whose Lipschitz constant is relatively small compared to $R$, so that 
$u^R$ is a viscosity solution of \erf{eHJ} as well. 

The existence and uniqueness (including comparison) of a viscosity solution of \erf{eHJR}, with initial condition 
$u(\cdot,0)=u_0$,  is a consequence of the classical theory of viscosity solutions. 
Let $u^R$ be such a solution and we seek for an estimate on the Lipschitz bound for $u^R$ that is 
independent of $R$. It is easy to select a constant $C_0>0$ so that 
\beq\label{C_0}
|H(x,p)-c(H)| \leq C_0 \ \ \text{ for all }\ x\in M,\ p\in B_\gk,
\eeq
where $\gk=\Lip(u_0)$ and $B_\gk$ denotes the ball $\subset T_x^*M$ 
of radius $\gk$ with center at the origin. It is then easy to see that the functions 
$(x,t)\mapsto u_0(x)-C_0t$ and $(x,t)\mapsto u_0(x)+C_0t$ are viscosity sub- and super-solutions of \erf{eHJR}, respectively. By the comparison principle, we get 
\beq\label{est-uR}
u_0(x)-C_0t\leq u^R(x,t)\leq u_0(x)+C_0t.
\eeq
Fix any $\tau>0$ and consider the function $v\mid (x,t)\mapsto u^R(x,t+\tau)$ in $M\tim[0,\,\infty)$, 
which is a viscosity solution of \erf{eHJR}.  
By \erf{est-uR}, we have\ $|v(x,0)-u^R(x,0)|\leq C_0\tau$ for $x\in M$. By the comparison principle
applied to the functions $u^R(x,t+\tau)$ and $u^R(x,t)\pm C_0\tau$,  
we see that $|v(x,t)-u^R(x,t)|\leq C_0\tau$ for all $(x,t)\in M\tim[0,\,\infty)$ and, hence, the function 
$u^R$ is Lipschitz continuous in $t$, with Lipschitz bound less or equal to $C_0$, where $C_0$ is independent of choice of $R$. Furthermore, for each $t>0$, 
$x\mapsto u^R(x,t)$ is a viscosity subsolution of $H_R(x,Du)=C_0$ in $M$. If $R>C_0$, then, by the definition of $\Gth_R$, this implies that $x\mapsto u^R(x,t)$ is a viscosity subsolution of $H(x,Du)=c(H)$ in $M$ 
for any $t>0$, which, together with the coercivity assumption (H3),
yields a Lipschitz bound, independent of $R$, of $u^R(x,t)$ as functions of $x$, uniform in $t$.  
Moreover, we deduce that, if $R>C_0$, then $u^R$ is a viscosity solution of \erf{eHJ} as well.  

To conclude, we select $C_0>0$ so that \erf{C_0} is satisfied, fix an $R>C_0$, and solve \erf{eHJR}. 
Then the solution $u^R$ is a viscosity solution of \erf{eHJ} and, moreover, it is Lipschitz 
continuous in $M\tim[0\,\,\infty)$. 
\eproof

\bproof[Outline of proof of Proposition \ref{comp-cp}] We need only to consider the 
case when $T<\infty$. 
We follow the proof of Theorem \ref{The_comp_prin} with minor modifications. 
It is enough to show that $v\leq w+\ep$ in $M\tim(0,\,T)$ for all $\ep>0$. Thus, we may assume 
by adding a positive constant to $w$ that $v(x,0)<w(x,0)$ for $x\in M$. It is then possible to select 
a function $u_0\in\Lip(M)$ so that $v(x,0)\leq u_0(x)\leq w(x,0)$ for $x\in M$. By Proposition \ref{exist-cp} there exists a viscosity solution $u\in \Lip(M\tim[0,\,\infty))$ of \erf{eHJ} satisfying the initial condition $u(\cdot,0)=u_0$. We need to prove that $v\leq u\leq w$ in $M\tim[0,\,T)$. 

The next step is to show that $v\leq u$ in $M\tim(0,\,T)$. The argument for proving 
the inequality $u\leq w$ is similar and we skip it here. 
We argue by contradiction and thus suppose that $\sup_{M\tim[0,\,T)}(v-u)>0$.   
We choose an $S\in(0,\,T)$, so close to $T$, that $\sup_{M\tim[0,\,S]}(v-u)>0$.  
Note that $v$ is bounded above in $M\tim [0,\,S]$ since $v$ is real-valued 
and upper semicontinuous in $M\tim[0,\,T)$. For $\gd>0$ we consider the function 
$\Psi_\gd\mid (x,t)\mapsto v(x,t)-u(x,t)-\gd(T-t)^{-1}$, which attains a maximum at a point $(x_\gd,t_\gd)
\in M\tim[0,\,S)$. Observe that, if $\gd>0$ is small enough, then the maximum value of $\Psi_\gd$ is positive,  and that the function $v_\gd\mid (x,t)\mapsto v(x,t)-\gd(T-t)^{-1}$ is a viscosity 
subsolution of $\pl_t v_\gd+H(x, Dv_\gd)=c(H)-\gd T^{-2}$ in $M\tim(0,\,S)$. We fix such a small $\gd$ 
in what follows.  Note that, since $v(x,0)\leq u(x,0)$, we have $t_\gd>0$. 
We fix a $\gf\in C^1(M\tim[0,\,S])$ so that $\gf(x_\gd,t_\gd)=0$, $D\gf(x_\gd,t_\gd)=\partial_t\gf(x_\gd,t_\gd)=0$, 
and $\gf(x,t)>0$ for all $(x,t)\not=(x_\gd,t_\gd)$. The function $\Psi_\gd-\gf$ achieves a strict maximum 
at $(x_\gd,t_\gd)$. 

We are now ready to apply the argument of doubling variables. Passing to local coordinates   
around $x_\gd$, we may assume that $x_\gd\in D$ for some open subset  $D$ of $\R^n$
such that  $\ol D\subset M$. We choose $\rho>0$ so that $[t_\gd-\rho, t_\gd+\rho] \subset (0,\,S)$. 
For $\ga>0$ we consider the function 
\[
\Phi_\ga(x,t,y,s)=(v_\gd-\gf)(x,t) -u(y,s)-\ga(|x-y|^2+(t-s)^2) 
\]
in $\ol D\tim [t_\gd-\rho,\,t_\gd+\rho] \tim \ol D\tim[t_0-\rho,\,t_\gd+\rho]$. 
Since $\Phi_\ga$ is upper semicontinuous, $\Phi_\ga$ achieves a maximum at a point $(x_\ga,t_\ga,y_\ga,s_\ga)$. Since $u$ is Lipschitz continuous, the inequality $\Phi_\ga(x_\ga.t_\ga,y_\ga,s_\ga)$ $\geq \Phi_\ga(x_\ga,t_\ga,x_\ga,t_\ga)$ yields 
\[
\ga(|x_\ga-y_\ga|^2+(t_\ga-s_\ga)^2)\leq \Lip(u)(|x_\ga-y_\ga|^2+|t_\ga-s_\ga|^2)^{1/2}, 
\]
which implies 
\[
\ga(|x_\ga-y_\ga|^2+|t_\ga-s_\ga|^2)^{1/2}\leq \Lip(u). 
\]
This shows that the collections $\{\ga(x_\ga-y_\ga)\}_{\ga>0}\subset \R^n$ and 
$\{\ga(t_\ga-s_\ga)\}_{\ga>0}\subset \R$ are bounded, and, in particular, for some 
sequence $\ga_j\to+\infty$, 
the sequence $\{(x_{\ga_j},t_{\ga_j},\ga_j(x_{\ga_j}-y_{\ga_j}),\ga_j(t_{\ga_j}-s_{\ga_j}))\}_{j\in\N}
\subset \R^n\tim\R\tim\R^n\tim\R$ is convergent. Set
\[
x_0=\lim_{j\to\infty} x_{\ga_j},\quad t_0=\lim_{j\to\infty} t_{\ga_j},\quad
p_0=\lim_{j\to\infty}2\ga_j(x_{\ga_j}-y_{\ga_j}),\quad q_0=\lim_{j\to\infty} 2\ga_j(t_{\ga_j}-s_{\ga_j}), 
\] 
and note that 
\[
\lim_{j\to\infty}y_{\ga_j}=x_0 \quad\and \quad\lim_{j\to\infty}s_{\ga_j}=t_0. 
\]
From the inequality
\[\bald
(\Psi_\gd-\gf)(x_\gd,t_\gd)&\,\leq\max_{(x,t)\in\ol D\tim[t_\gd-\rho,t_\gd+\rho]}\Phi_\ga(x,t,x,t)
= \Phi_\ga(x_\ga,t_\ga,y_\ga,s_\ga)
\\&\,\leq (v_\gd-\gf)(x_\ga,t_\ga)-u(y_\ga,s_\ga),
\eald\]
using the upper semicontinuity of $v$, we obtain 
\[
(\Psi_\gd-\gf)(x_\gd,t_\gd)\leq (v-\gf)(x_0,t_0)-u(x_0,t_0)=(\Psi_\gd-\gf)(x_0,t_0),
\] 
which ensures, since $(x_\gd,t_\gd)$ is a strict maximum point of $\Psi_\gd-\gf$, that 
$(x_0,t_0)=(x_\gd,t_\gd)$. Thus, for sufficiently large $j$, we have $x_{\ga_j},y_{\ga_j}\in D$ 
and $t_{\ga_j}, s_{\ga_j}\in(t_\gd-\rho,t_\gd+\rho)$. For such $j$, by the viscosity properties of $v$ and $u$, 
we have 
\[
2\ga_j(t_{\ga_j}-s_{\ga_j})+H(x_{\ga_j},D\gf(x_{\ga_j},t_{\ga_j})+2\ga(x_{\ga_j}-y_{\ga_j}))\leq c(H)-\gd T^{-2},
\]
and 
\[
2\ga_j(t_{\ga_j}-s_{\ga_j})+H(y_{\ga_j},2\ga(x_{\ga_j}-y_{\ga_j}))\geq c(H).
\]
Moreover, in the limit as $j\to\infty$, we obtain
\[
q_0+H(x_0,p_0)\leq c(H)-\gd T^{-2}\quad\and\quad q_0+H(x_0,p_0)\geq c(H).
\]
These yield a contradiction. 
\eproof

\bibliographystyle{abbrv}
\bibliography{mybib}

\end{document}